\documentclass[a4paper,12pt,reqno]{amsart}
\usepackage{lmodern,amsmath,nccmath,amssymb,mathtools,amsthm}
\usepackage{enumerate,cite}
\usepackage{graphics,graphicx,geometry}
\usepackage{hyperref}
\usepackage{xcolor}
\usepackage{mathrsfs}
\usepackage{wasysym}
\usepackage{cleveref}
\usepackage{floatrow,morefloats}
\usepackage{epsfig}
\usepackage{epstopdf}
\usepackage{caption,subcaption}
\usepackage{exscale, relsize}

\newcommand*{\fminus}{\genfrac{}{}{0pt}{}{}{-}}

\numberwithin{equation}{section}

\newtheorem{definition}{Definition}

\newtheorem{lemma}{Lemma}
\newtheorem{theorem}{Theorem}
\newtheorem{proposition}{Proposition}

\newtheorem{observation}{Observation}

\numberwithin{equation}{section}
\numberwithin{definition}{section}
\numberwithin{lemma}{section}
\numberwithin{theorem}{section}
\numberwithin{corollary}{section}
\numberwithin{remark}{section}
\numberwithin{example}{section}

\title{Recovering orthogonality from quasi-nature of Spectral transformations}

\author{VIKASH KUMAR$^\dagger$}
\address{$^\dagger$Department of Mathematics\\ Indian Institute of Technology, Roorkee-247667, Uttarakhand, India}
\email{vikaskr0006@gmail.com, vkumar4@mt.iitr.ac.in}

\author{Francisco Marcell\'an$^{\dagger\dagger}$}
\address{$^{\dagger\ddagger}$ Departamento de Matem\'aticas\\ Universidad Carlos III de Madrid, Legan\'es, Spain}
\email{pacomarc@ing.uc3m.es}

\author{A. Swaminathan$^{\dagger\dagger\dagger}$}
\address{$^{\ddagger\dagger\dagger}$Department of Mathematics\\ Indian Institute of Technology, Roorkee-247667, Uttarakhand, India}
\email{mathswami@gmail.com, a.swaminathan@ma.iitr.ac.in}
\allowdisplaybreaks
\bigskip
\begin{document}

	\begin{abstract}
		In this contribution, quasi-orthogonality of polynomials generated by Geronimus and Uvarov transformations is analyzed. An attempt is made to discuss the recovery of the source orthogonal polynomial from the quasi-Geronimus and quasi-Uvarov polynomials of order one. Moreover, the discussion on the difference equation satisfied by quasi-Geronimus and quasi-Uvarov polynomials is presented. Furthermore, the orthogonality of quasi-Geronimus and quasi-Uvarov polynomials is achieved through the reduction of the degree of coefficients in the difference equation. During this procedure, alternative representations of the parameters responsible for achieving orthogonality are derived. One of these representations involves the Stieltjes transform of the measure. Finally, the recurrence coefficients ensuring the existence of a measure that makes the quasi-Geronimus Laguerre polynomial of order one an orthogonal polynomial are calculated.
		\end{abstract}
	\maketitle
	\markboth{Vikash Kumar, F. Marcell\'an and A. Swaminathan}{Quasi-spectral transformations}

	\section{Introduction}
	Let $\mathcal{L}$ be a quasi-definite linear functional in the linear space of polynomials with complex coefficients  such that their moments are finite complex numbers. Let  $\{P_n(x)\}_{n=0}^{\infty}$ be a sequence of monic orthogonal polynomials  with respect to $\mathcal{L}$. Then there exist  sequences of complex numbers  $\{c_{n}\}_{n=1}^{\infty}$ and $\{\lambda_{n}\}_{n=1}^{\infty},$ $\lambda_{n}\neq0, n\geq1,$  such that $P_n(x)$ satisfies the three term recurrence relation (TTRR, in short)
	\begin{align}
		xP_n(x)=P_{n+1}(x)+c_{n+1}P_n(x)+\lambda_{n+1}P_{n-1}(x),~~n\geq0,
	\end{align}
with $P_0(x)=1$ and $P_{-1}(x)=0$. Note that $\lambda_1$ can be chosen arbitrary. Also, if $\mathcal{L}$ is positive-definite, then $c_n\in\mathbb{R}$ and $\lambda_{n}>0, n\geq 1,$ see \cite{Chihara book}.\\

The exploration of quasi-orthogonal polynomials traces back to Riesz's work \cite{Riesz} in 1923, where he introduced the notion of linear combinations of consecutive elements of a sequence of orthogonal polynomials, termed quasi-orthogonal polynomials of order one. Riesz applied this concept in the proof of the Hamburger moment problem. Fourteen years later, Fejér \cite{Fejer} delved into the study of linear combinations involving three consecutive elements of orthogonal polynomials. Shohat \cite{Shohat} extended Fejér's results and introduced the concept of finite linear combinations of orthogonal polynomials with constant coefficients in the examination of mechanical quadrature formulas.
In the process of self-perturbation of  orthogonal polynomials, we let go of their usual orthogonality within the sequence of polynomials. This particular aspect is explored in \cite{Alfaro2011 CMA,When do linear}, where the discussion revolves around the orthogonality of quasi-orthogonal polynomials. They tackle this by putting constraints on the choices of constant coefficients used in the linear combination of orthogonal polynomials. Furthermore, \cite{Ismail_2019_quasi-orthogonal} discusses the difference equation fulfilled by the sequence of quasi-orthogonal polynomials of order one and investigates the orthogonality of these polynomials using the spectral theorem. For a deeper understanding of quasi-orthogonal polynomials, we recommend referring to works such as \cite{Chihara book, Chihara, Akhiezer 1965, Dickinson, Draux1990, Draux2016,Shukla_Swami_Racsam_2023}.

\begin{definition}\label{Quasi-orthogonal definition}{\rm\cite{Chihara book}}
	A polynomial $p(x)$ of degree $n$ is said to be quasi-orthogonal polynomial of order one with respect to the quasi-definite linear functional $\mathcal{L}$ if
	\begin{align*}
		\mathcal{L}(x^jp(x))=\int x^jp(x)d\mu=0, ~~  j=0,1,...,n-2.
	\end{align*}
\end{definition}
According to Definition \ref{Quasi-orthogonal definition}, we can easily deduce that $P_{n}(x)$ and $P_{n-1}(x)$ are both quasi-orthogonal polynomials of order one.  The necessary and sufficient condition, as per \cite{Chihara book}, for a polynomial to be a quasi-orthogonal polynomial of order one is the linear combination of $P_{n}(x)$ and $P_{n-1}(x)$ with constant coefficients, where the coefficients cannot be zero simultaneously.\\
\subsection{Motivation of the problem}
In the study of orthogonal polynomials, problems can be approached as inverse problems through various methods. A notable example is Favard's theorem \cite[Theorem 4.4]{Chihara book}. This theorem establishes the existence of a quasi-definite linear functional such that the sequence of monic polynomials  defined by a TTRR with appropriate recurrence coefficients becomes orthogonal.\\

An intriguing problem arises when considering a sequence of orthogonal polynomials $\{P_n(x)\}_{n=0}^{\infty}$ with respect to a quasi-definite linear functional. Given another sequence of polynomials $\{Q_n(x)\}_{n=0}^{\infty}$, such that
\begin{align*}
	Q_n(x)+\sum_{l=1}^{m-1}\alpha_{l,n}Q_{n-l}(x)=P_n(x)+\sum_{l=1}^{j-1}\beta_{l,n}P_{n-l}(x)
\end{align*} holds, to find necessary and sufficient conditions in order to $\{Q_n(x)\}_{n=0}^{\infty}$ be orthogonal. The relation between these polynomials and their corresponding linear functionals is then explored as an inverse problem. This investigation is conducted for various pairs $(m,j)$ and has been addressed in \cite{Alfaro_Paco_Linear related OP_JMAA_2003,Paco_Petronilho_1995_M-N type problem,When do linear,Alfaro2011 CMA} and related references. It is noteworthy that when $m=1$ and $j=k$, the result corresponds to quasi-orthogonal polynomials of order $k$, see \cite{Bracciali}.\\

The expression of orthogonal polynomials in terms of quasi-orthogonal polynomials of order one using spectral transformations is discussed in \cite{Derevyagin_Bailey_CJM by DT}. Additionally, the study in \cite{Vikas_Swami_quasi-type kernel} explores the recovery of orthogonal polynomials from quasi-type kernel polynomials of order one. This manuscript addresses the inverse problem, aiming to reconstruct the original orthogonal polynomial from weak orthogonality. We introduce the quasi-Geronimus polynomial of order one and quasi-Uvarov polynomial of order one, both possessing a quasi nature that adds intrigue to the recovery process. The methodology involves forming linear combinations of quasi-Geronimus polynomials with polynomials generated by linear spectral transformations with rational coefficients. Essential to establishing orthogonality is the calculation of sequences of constants. Throughout this process, the three-term recurrence relation satisfied by orthogonal polynomials and the linearly independent nature of the set $\{P_0(x), P_1(x),...,P_n(x)\}$ play pivotal roles. More detailed proofs can be found in Section \ref{Rec: quasi-Geronimus and recovery} and Section \ref{Rec: quasi-Uvarov and recovery}.\\

\subsection{Organization} In Section \ref{Spectral transformations}, we explore linear spectral transformations and their associated orthogonal polynomials. In Section \ref{Rec: quasi-Geronimus and recovery}, we introduce the concept of quasi-Geronimus polynomial of order one and demonstrate the recovery of the source orthogonal polynomial through a linear combination of consecutive degrees of quasi-Geronimus polynomial of order one. Moreover, we delve into various representations of source orthogonal polynomials in relation to quasi-Geronimus polynomial of order one and the polynomials generated by linear spectral transformations. In Section \ref{Rec: quasi-Uvarov and recovery}, we focus the attention on the quasi-Uvarov polynomial of order one and explore its orthogonality. In addition, we demonstrate the recovery of the source orthogonal polynomial through the consecutive degrees of quasi-Uvarov polynomial of order one. Furthermore, employing a similar approach as in Section \ref{Rec: quasi-Geronimus and recovery}, we express the orthogonal polynomials as a linear combination of quasi-Uvarov polynomial of order one and the polynomial generated by linear spectral transformations. In Section \ref{Numerical experiment}, we derive the difference equation for quasi-Geronimus polynomial of order one as well as quasi-Uvarov polynomial of order one corresponding to the initial polynomial being Laguerre polynomial. Additionally, the closed form of $\beta_n$ satisfying \eqref{beta n restriction cond}, a necessary condition for the orthogonality of quasi-Geronimus Laguerre polynomial of order one, is determined. Subsequently, the recurrence parameters are calculated to ensure the existence of an orthogonality measure.
 Finally, we present numerical experiments on the zeros of the quasi-Geronimus Laguerre polynomials.

\section{Linear spectral transformations for orthogonal polynomials}\label{Spectral transformations}
	Perturbation techniques play a crucial role in the study of the theory of orthogonal polynomials. Since the foundational work of Christoffel, and notably in recent years, Marcellán and his collaborators have been significant contributors to this field. A recent book by García-Ardila, Marcellán, and Marriaga \cite{Paco Book_2021} focuses the attention on orthogonal polynomials on the real line, providing a thorough discussion of some perturbations, the so called linear spectral transformations, of a linear functional. The three essential linear spectral transformations—Christoffel, Geronimus, and Uvarov—can be achieved through modifications of the linear functional, see also \cite{Alexei}. To enhance reader understanding, we offer a detailed exploration of these spectral transformations and their corresponding orthogonal polynomials.
	\subsection{Christoffel transformation} Suppose $\mathcal{L}$ is a quasi-definite linear functional and let $\{P_n(x)\}_{n=0}^\infty$ be its  corresponding sequence of monic orthogonal polynomials.  We can define the generalized Christoffel transformation by multiplying the linear functional by a fixed degree polynomial. In particular, we define the {\it canonical Christoffel transformation at $a\in\mathbb{R}$} by multiplying the linear functional by a polynomial of degree 1. The new linear functional denoted by $\mathcal{L}^C$ is defined by
	\begin{align*}
		\mathcal{L}^C[p(x)]=\mathcal{L}[(x-a)p(x)],
	\end{align*}
for any polynomial $p(x)$. In the positive-definite case, if $a$ lies outside the interior of the convex hull of the support of a measure associated with the linear functional $\mathcal{L}$, that is $P_n(a)\neq0$ for any $n\in\mathbb{N}\cup{\{0\}}$, then it ensures the existence of orthogonal polynomials with respect to $\mathcal{L}^C$. If the linear functional $\mathcal{L}$ is quasi-definite, then a necessary and sufficient condition for  the quasi-definiteness of $\mathcal{L}^C$ is $P_n(a)\neq0, n\geq 1,$ as well as $\mathcal{L}^C[1]\neq0$. The sequence of monic orthogonal polynomials $\{\mathcal{C}_n(x;a)\}_{n=0}^{\infty}$ corresponding to such a canonical Christoffel transformation are \cite{ Chihara book}, \cite{Derevyagin_Bailey_CJM by DT}
\begin{equation}\label{kernel}
	\mathcal{C}_n(x;a)=\frac{1}{x-a}\left[P_{n+1}(x)-\frac{P_{n+1}(a)}{P_n(a)} P_n(x)\right], ~ n\geq 0.
\end{equation}
The polynomial $\mathcal{C}_n(x;a)$ corresponding to $\mathcal{L}^C$ is known as a monic kernel polynomial, see  \cite{Chihara book}.
 They also satisfy the TTRR
\begin{align}\label{Kernel-TTRR}x\mathcal{C}_n(x;a)=\mathcal{C}_{n+1}(x;a)+c_{n+1}^c\mathcal{C}_n(x;a)+ \lambda_{n+1}^c\mathcal{C}_{n-1}(x;a), ~ n\geq 0,
\end{align}
where
\begin{align}\label{Kernel recurrence parameters}
	\lambda_{n}^c=\lambda_{n}\frac{P_{n}(a)P_{n-2}(a)}{P_{n-1}^2(a)}, n\geq 2, \quad c_n^c=c_{n+1}-\frac{P_{n}^2(a)-P_{n-1}(a)P_{n+1}(a)}{P_{n-1}(a)P_{n}(a)}, ~n\geq 1.
\end{align}
Moreover, the Christoffel-Darboux formula \cite[eq. 4.9]{Chihara book} holds
\begin{align}{\label{CD identity}}
	\lambda_1\lambda_2...\lambda_{n+1}\sum_{j=0}^{n}\frac{P_j(x)P_j(a)}{\lambda_1\lambda_2...\lambda_{j+1}}=\frac{P_{n+1}(x)P_n(a)-P_{n+1}(a)P_n(x)}{x-a}, ~n\geq0.
\end{align}
Using \eqref{CD identity}, we can write the monic kernel polynomials as
\begin{align}\label{Other form of kernel polynomial}
	\mathcal{C}_n(x;a)=\lambda_1\lambda_2...\lambda_{n+1}(P_n(a))^{-1}\mathcal{K}_n(x,a),
\end{align}
where \begin{align}
	\mathcal{K}_n(x,a)=\sum_{j=0}^{n}\frac{P_j(x)P_j(a)}{\lambda_1\lambda_2...\lambda_{j+1}}.
\end{align}
\subsection{Geronimus transformation}
Let $\mathcal{L}$ be a quasi-definite linear functional. We define a linear functional by perturbing $\mathcal{L}$ in the sense of Geronimus. The new linear functional, known as the Geronimus transformation at $a\in\mathbb{R}$, is denoted by $\mathcal{L}^G$ is defined by
\begin{align}\label{Geronimus linear funtional}
	\mathcal{L}^G[p(x)]=\mathcal{L}[\frac{p(x)- p(a)}{x-a}]+Mp(a)
\end{align}
for any polynomial $p(x),$ see \cite{Maroni}. Since the inclusion of the arbitrary constant $M,$ the canonical Geronimus transformation is not uniquely defined. Furthermore, it can be observed that $\mathcal{L}^G(1)=M$. Suppose $\mathcal{L}^G$ is a quasi-definite linear functional. In that case, there exists a sequence of monic orthogonal polynomials denoted by  $\{\mathcal{G}_n(x;a)\}_{n=0}^{\infty}$ corresponding to the canonical Geronimus transformation. They are given by
\begin{align}\label{Geronimus polynomial}
	\mathcal{G}_n(x;a)=P_n(x)+\chi_n(a)P_{n-1}(x), ~n\geq1,
\end{align}
where
\begin{align}
	\chi_n(a)=-\frac{\mathcal{L}(1)Q_{n-1}(a)+MP_n(a)}{\mathcal{L}(1)Q_{n-2}(a)+MP_{n-1}(a)}, ~n\geq 1,
\end{align}
and  the sequence of polynomials  $\{Q_{n}(x)\}_{n=0}^{\infty}$ is known in  the literature as either numerator polynomials (see \cite{Chihara book}) or associated polynomials of the first kind,  of degree $n-1$.
The polynomial corresponding to $\mathcal{L}^G$ is termed the Geronimus polynomial. It is essential to note that the necessary and sufficient conditions for $\mathcal{L}^G$ to be quasi-definite are $M\neq0$ and $\mathcal{L}(1)Q_{n-1}(a)+MP_n(a)\neq0, n\geq1.$ The TTRR satisfied by Geronimus polynomials is given by
\begin{align}\label{Geronimus-TTRR}x\mathcal{G}_n(x;a)=\mathcal{G}_{n+1}(x;a)+c_{n+1}^g\mathcal{G}_n(x;a)+ \lambda_{n+1}^g\mathcal{G}_{n-1}(x;a), ~ n\geq0,
\end{align}
with
\begin{align*}
	c_{n+1}^g=c_{n+1}-\chi_{n}(a)+\chi_{n+1}(a), n\geq0 ,~~~~\lambda_{n+1}^g=\lambda_{n}\frac{\chi_{n}(a)}{\chi_{n-1}(a)}, ~ n\geq1.
\end{align*}
\subsection{Uvarov transformation}
Suppose $\mathcal{L}$ is a quasi-definite linear functional. Uvarov \cite{Uvarov 1969_Uvarov tranformation} introduced a new linear functional as a perturbation of  $\mathcal{L}$ by the addition of a finite number of point masses. In particular, the canonical Uvarov transformation is defined by adding one point mass. The new linear functional denoted by $\mathcal{L}^U$ is defined as
\begin{align*}
	\mathcal{L}^U[p(x)]=\mathcal{L}[p(x)]+Mp(a),
\end{align*}
for any polynomial $p(x)$. If $\mathcal{L}^U$ is a quasi-definite linear functional, then the corresponding sequence $\{\mathcal{U}_n(x;a)\}_{n=0}^{\infty}$  of monic orthogonal polynomials is given by
\begin{align*}
	\mathcal{U}_n(x;a)=P_n(x)-t_n\mathcal{C}_{n-1}(x;a), ~n\geq 1,
\end{align*}
where
\begin{align*}
	t_n=\frac{MP_n(a)P_{n-1}(a)}{\lambda_1\lambda_2...\lambda_{n}\left(1 + M\mathcal{K}_{n-1}(a,a)\right)}, ~n\geq1.
\end{align*}
The necessary and sufficient condition for quasi-definiteness of $\mathcal{L}^U$ is $M\neq -(\mathcal{K}_{n-1}(a,a))^{-1}$ for $n\geq1$. The polynomials corresponding to $\mathcal{L}^U$ are referred to as the Uvarov polynomials. Since $\{\mathcal{U}_n(x;a)\}_{n=0}^{\infty}$ constitutes a sequence of monic orthogonal polynomials it satisfies a TTRR given by
\begin{align}\label{Uvarov-TTRR}x\mathcal{U}_n(x;a)=\mathcal{U}_{n+1}(x;a)+c_{n+1}^u\mathcal{U}_n(x;a)+ \lambda_{n+1}^u\mathcal{U}_{n-1}(x;a), ~ n\geq 0,
\end{align}
with
\begin{align*}
	c_{n+1}^u=c_{n+1}-t_n+t_{n+1}, ~~~~\lambda_{n+1}^u=\lambda_{n}\frac{\lambda_{n+1}+t_{n}\frac{P_n(a)}{P_{n-1}(a)}}{\lambda_{n}+t_{n-1}\frac{P_{n-1}(a)}{P_{n-2}(a)}}.
\end{align*}

	\section{Recovery from quasi-Geronimus polynomial of order one}\label{Rec: quasi-Geronimus and recovery}
	We observe that the Geronimus polynomial is obtained by perturbing the linear functional $\mathcal{L}$. In this section, we self-perturb the Geronimus polynomial and introduce the concept of the quasi-Geronimus polynomial of order one. We characterize the quasi-Geronimus polynomial of order one and discuss its orthogonality. The section concludes by recovering the source orthogonal polynomials.
	\begin{definition}
		Let $\mathcal{L}^G$ be the Geronimus transformation of $\mathcal{L}$ at $a$. Let $\{\mathcal{G}_n(x;a)\}_{n=0}^{\infty}$ be the sequence of Geronimus polynomials which is orthogonal with respect to  $\mathcal{L}^G$. A polynomial $p$ is said to be quasi-Geronimus polynomial of order one if it is of degree at most $n$ and satisfies
		\begin{align*}
			\mathcal{L}^G(x^kp(x))=0~~for~~ k=0,1,2,...,n-2.
		\end{align*}
	\end{definition}
Since $\{\mathcal{G}_n(x;a)\}_{n=0}^{\infty}$ is a sequence of orthogonal polynomials with respect to $\mathcal{L}^G$, then the Geronimus polynomial of degree $n+1$ and $n$ are quasi-Geronimus polynomial of order one. The subsequent result characterizes the quasi-Geronimus polynomial of order one as a self-perturbation of Geronimus polynomials.
\begin{lemma}
	A polynomial $\mathcal{G}_n^Q(x;a)$ of degree $n$ is a quasi-Geronimus polynomial of order one if and only if $\mathcal{G}^Q(x;a)$ can be written as
	\begin{align}\label{Non monic Quasi-Geronimus order 1}
		\mathcal{G}_n^Q(x;a)=b_n\mathcal{G}_n(x;a)+\beta_n\mathcal{G}_{n-1}(x;a),
	\end{align}
where coefficients $b_n$ and $\beta_n$ cannot be zero simultaneously.
\end{lemma}
\begin{proof} See \cite{Chihara book}.

\end{proof}
In \cite[Proposition 3]{Vikas_Swami_quasi-type kernel}, it is shown that the source orthogonal polynomial $P_n(x)$ can be expressed as a linear combination of $\mathcal{G}_{n+1}(x;a)$ and $\mathcal{G}_{n}(x;a)$:
	\begin{align*}
(x-a)P_n(x)=\mathcal{G}_{n+1}(x;a)-\frac{\lambda_{n+1}}{\chi_n(a)}\mathcal{G}_{n}(x;a), ~ n\geq0.
\end{align*}
As the sequence of monic quasi-Geronimus polynomials of order one is not an orthogonal system with respect to a linear functional, it does not satisfy a TTRR. However, Theorem \ref{Difference equation for Quasi-Geronimus} demonstrates that the sequence still follows a difference equation with linear and quadratic coefficients. To establish this, we use Lemma \ref{Original in terms of quasi Geronimus}, where we express the source monic orthogonal polynomial $P_n(x)$ in terms of monic quasi-Geronimus polynomials of order one with variable coefficients.

\begin{lemma}\label{Original in terms of quasi Geronimus} Let $\{P_n(x)\}_{n=0}^\infty$ be a sequence of monic orthogonal polynomials with respect to $\mathcal{L}$ and $\mathcal{G}_{n}^{Q}(x;a)$ be a monic quasi-Geronimus polynomial of order one. Then there exist polynomials $\displaystyle{l_n(x)}$ and $\displaystyle{j_n(x)}$ such that
	\begin{align*}
		j_n(x)P_n(x)=d_n(x)\mathcal{G}_{n+1}^{Q}(x;a)+(c_{n+1}+\lambda_{n+1}-\chi_{n}(a)\beta_{n+1})\mathcal{G}_{n}^{Q}(x;a), ~n \geq 0,
	\end{align*}
where $\displaystyle{l_n(x)=x-c_{n+1}+\chi_{n+1}(a)+\beta_{n+1},~ d_n(x)=\chi_n(a)+\beta_n+(x-c_n)\chi_{n-1}(a)\frac{\beta_n}{\lambda_{n}}}$\\ and
$\displaystyle{j_n(x)=l_n(x)d_n(x)-(c_{n+1}+\lambda_{n+1}-\chi_{n}(a)\beta_{n+1})\left(	\chi_{n-1}(a)\frac{\beta_n}{\lambda_{n}}-1\right)}$.
\end{lemma}
\begin{proof}
According to \ref{Geronimus polynomial}, we can write
\begin{align*}
	\mathcal{G}_{n}^{Q}(x;a)=P_{n}(x)+(\chi_{n}(a)+\beta_{n})P_{n-1}(x)+\beta_{n}\chi_{n-1}(a)P_{n-2}(x).
\end{align*}
By using the expansion of $x P_{n-1}(x)$ we obtain
\begin{align*}
	\mathcal{G}_{n}^{Q}(x;a)=\left(1-\chi_{n-1}(a)\frac{\beta_n}{\lambda_{n}}\right)P_n(x)+\left(\chi_n(a)+\beta_n+x\chi_{n-1}(a)\frac{\beta_n}{\lambda_{n}}-c_n\chi_{n-1}(a)\frac{\beta_n}{\lambda_{n}}\right)P_{n-1}(x).
\end{align*}
Similarly, one can use the expansion of $xP_{n}(x)$ to write
\begin{align*}
	\mathcal{G}_{n+1}^{Q}(x;a)=(x-c_{n+1}+\chi_{n+1}(a)+\beta_{n+1})P_n(x)+(\chi_{n}(a)\beta_{n+1}-c_{n+1}-\lambda_{n+1})P_{n-1}(x).
\end{align*}
As a consequence, the transfer matrix from $P_{n}(x)$ and $P_{n-1}(x)$ to  $\mathcal{G}_{n+1}^{Q}(x;a)$ and $\mathcal{G}_{n}^{Q}(x;a)$ is
\[\begin{pmatrix}
	\mathcal{G}_{n+1}^{Q}(x;a)\\
	\mathcal{G}_{n}^{Q}(x;a)
\end{pmatrix}=\begin{pmatrix}
	l_n(x) & \chi_{n}(a)\beta_{n+1}-c_{n+1}-\lambda_{n+1}\\
	1-\chi_{n-1}(a)\frac{\beta_n}{\lambda_{n}} &d_n(x)
\end{pmatrix}\begin{pmatrix}
	P_{n}(x)\\
	P_{n-1}(x)
\end{pmatrix},\]
where
$\displaystyle{l_n(x)=x-c_{n+1}+\chi_{n+1}(a)+\beta_{n+1},~ d_n(x)=\chi_n(a)+\beta_n+(x-c_n)\chi_{n-1}(a)\frac{\beta_n}{\lambda_{n}}}$.
Since the above matrix is nonsingular, we write
\begin{align}\label{Matrix form polynomial in terms of Quasi Geronimus}
	j_n(x)\begin{pmatrix}
		P_{n}(x)\\
		P_{n-1}(x)
	\end{pmatrix}=\begin{pmatrix}
		d_n(x)	 & c_{n+1}+\lambda_{n+1}-\chi_{n}(a)\beta_{n+1}\\
		\chi_{n-1}(a)\frac{\beta_n}{\lambda_{n}}-1 &l_n(x)
	\end{pmatrix}\begin{pmatrix}
		\mathcal{G}_{n+1}^{Q}(x;a)\\
		\mathcal{G}_{n}^{Q}(x;a)
	\end{pmatrix},
\end{align}
where $\displaystyle{j_n(x)=l_n(x)d_n(x)-(c_{n+1}+\lambda_{n+1}-\chi_{n}(a)\beta_{n+1})\left(	\chi_{n-1}(a)\frac{\beta_n}{\lambda_{n}}-1\right)}$. This completes the proof.
\end{proof}

\begin{theorem}\label{Difference equation for Quasi-Geronimus}Let $\{P_n(x)\}_{n=0}^\infty$ and $\{\mathcal{G}_{n}(x;a)\}_{n=0}^{\infty}$ be the sequences of monic orthogonal polynomials with respect to $\mathcal{L}$ and $\mathcal{L}^G,$ respectively. Then the difference equation satisfied by monic quasi-Geronimus polynomials of order one is
	\begin{align*}
		j_n(x)\mathcal{G}_{n+2}^{Q}(x;a)&=\left(d_n(x)m_{n+1}(x)-\lambda_{n+1}l_{n+1}(x)\left(\chi_{n-1}(a)\frac{\beta_n}{\lambda_{n}}-1\right)\right)\mathcal{G}_{n+1}^{Q}(x;a)\\
		&+\huge(m_{n+1}(x)(c_{n+1}+\lambda_{n+1}-\chi_{n}(a)\beta_{n+1})-\lambda_{n+1}l_{n}(x)l_{n+1}(x))\huge\mathcal{G}_{n}^{Q}(x;a), ~n\geq0,
	\end{align*}
where $\displaystyle{m_{n+1}(x):=l_{n+1}(x)(x-c_{n+1})+\beta_{n+2}\chi_{n+1}(a)-\lambda_{n+2}}$.
\end{theorem}
\begin{proof} We write
\begin{flalign*}
	\mathcal{G}_{n+2}^{Q}(x;a)&=P_{n+2}(x)+(\chi_{n+2}(a)+\beta_{n+2})P_{n+1}(x)+\beta_{n+2}\chi_{n+1}(a)P_{n}(x)\\
	&=(x-c_{n+2}+\chi_{n+2}(a)+\beta_{n+2})P_{n+1}(x)+(\beta_{n+2}\chi_{n+1}(a)-\lambda_{n+2})P_n(x)\\
	&=((x-c_{n+1})(x-c_{n+2}+\chi_{n+2}(a)+\beta_{n+2})+\beta_{n+2}\chi_{n+1}(a)-\lambda_{n+2})P_n(x)\\
	&\hspace{6cm}-\lambda_{n+1}(x-c_{n+2}+\chi_{n+2}(a)+\beta_{n+2})P_{n-1}(x)\\
	&=(l_{n+1}(x)(x-c_{n+1})+\beta_{n+2}\chi_{n+1}(a)-\lambda_{n+2})P_n(x)-\lambda_{n+1}l_{n+1}(x)P_{n-1}(x).
\end{flalign*}

Denoting $\displaystyle{m_{n+1}(x):=l_{n+1}(x)(x-c_{n+1})+\beta_{n+2}\chi_{n+1}(a)-\lambda_{n+2}}$ and using \eqref{Matrix form polynomial in terms of Quasi Geronimus} we obtain the desired result.
\end{proof}
When we subject the monic Geronimus polynomial to self-perturbation, the orthogonality condition is no longer preserved. However, despite this, we observe that it still satisfies the difference equation. To restore the full orthogonality of monic quasi-Geronimus polynomials of order one, we can impose conditions on $\beta_{n}$ in order to reduce the degree of coefficients in Theorem \ref{Difference equation for Quasi-Geronimus}. In such a way we get the recurrence parameters for monic quasi-Geronimus polynomials of order one that yields a TTRR.
\begin{proposition}\label{orthogonality of quasi-Geronimus}
	Let $\mathcal{G}_{n}^{Q}(x;a)$ be a monic quasi-Geronimus polynomial of order one with parameter $\beta_{n}$ such that
	\begin{align}\label{beta n restriction cond}
		\beta_{n}(c_{n+1}^g-c_n^g+\beta_{n}-\beta_{n+1})+\frac{\beta_{n}}{\beta_{n-1}}\lambda_{n}^g-\lambda_{n+1}^g=0, ~n\geq2.
	\end{align}
Then the polynomials $\mathcal{G}_{n}^{Q}(x;a)$ satisfy the TTRR
\begin{align*}
	\mathcal{G}_{n+1}^{Q}(x;a)-(x-c_{n+1}^{qg})\mathcal{G}_{n}^{Q}(x;a)+\lambda_{n+1}^{qg}\mathcal{G}_{n-1}^{Q}(x;a)=0, ~ n\geq 0,
\end{align*}
where the recurrence parameters are given by
\begin{align*}
	\lambda_{n+1}^{qg}=\frac{\beta_{n}}{\beta_{n-1}}\lambda_{n}^g, ~~~~~~~c_{n+1}^{qg}=c_{n+1}^g+\beta_{n}-\beta_{n+1}.
\end{align*}
If $\lambda_{n+1}^{qg}\neq 0$, then  $\{\mathcal{G}^Q_n(x;a)\}_{n=1}^{\infty}$ is an orthogonal polynomial sequence with respect to a quasi-definite linear functional. If  $\lambda_{n+1}^{qg} > 0$, then the corresponding linear functional is positive definite.
\end{proposition}
\begin{proof}We simplify
	\begin{align*}
	&\mathcal{G}_{n+1}^{Q}(x;a)-(x-c_{n+1}^{qg})\mathcal{G}_{n}^{Q}(x;a)+\lambda_{n+1}^{qg}\mathcal{G}_{n-1}^{Q}(x;a)
	=\mathcal{G}_{n+1}(x;a)\\
	&-(x-c_{n+1}^{qg}-\beta_{n+1})\mathcal{G}_{n}(x;a)-(\beta_{n}(x-c_{n+1}^{qg})-\lambda_{n+1}^{qg})\mathcal{G}_{n-1}(x;a)+\lambda_{n+1}^{qg}\beta_{n-1}\mathcal{G}_{n-2}(x;a).
\end{align*}
From the TTRR satisfied by Geronimus polynomials, we obtain
\begin{align*}
	&\mathcal{G}_{n+1}^{Q}(x;a)-(x-c_{n+1}^{qg})\mathcal{G}_{n}^{Q}(x;a)+\lambda_{n+1}^{qg}\mathcal{G}_{n-1}^{Q}(x;a)\\
	&=(c_{n+1}^{qg}-c_{n+1}^g+\beta_{n+1}-\beta_{n})\mathcal{G}_n(x;a)+(\beta_{n}c_{n+1}^{qg}+\lambda_{n+1}^{qg}-\lambda_{n+1}^{g}-\beta_{n}c_n^g)\mathcal{G}_{n-1}(x;a)\\
	&+(\lambda_{n+1}^{qg}\beta_{n-1}-\beta_{n}\lambda_{n}^g)\mathcal{G}_{n-2}(x;a).
	\end{align*}
 Since $\mathcal{G}_n(x;a)$, $\mathcal{G}_{n-1}(x;a)$ and $\mathcal{G}_{n-2}(x;a)$ are linearly independent, the left hand side is zero if and only if
 \begin{align*}
 	\beta_{n}(c_{n+1}^g-c_n^g+\beta_{n}-\beta_{n+1})+\frac{\beta_{n}}{\beta_{n-1}}\lambda_{n}^g-\lambda_{n+1}^g=0,
 \end{align*}
as well as
\begin{align*}
	\lambda_{n+1}^{qg}=\frac{\beta_{n}}{\beta_{n-1}}\lambda_{n}^g, ~~~~~~~c_{n+1}^{qg}=c_{n+1}^g+\beta_{n}-\beta_{n+1}.
\end{align*}
If $\lambda_{n+1}^{qg}\neq 0$, then from Favard's theorem there exists a quasi-definite  linear functional such that the sequence $\{\mathcal{G}^Q_n(x;a)\}_{n=1}^{\infty}$ becomes orthogonal. If $\lambda_{n+1}^{qg}> 0,$ then the linear functional is positive definite
\end{proof}
\subsection{An alternative representation of $\beta_{n}$}\label{Different representation of beta n} As observed in Proposition \ref{orthogonality of quasi-Geronimus}, the restrictions on the parameters $\beta_{n}$ play a crucial role in achieving the orthogonality of quasi-Geronimus polynomials. Therefore, exploring alternative representations of $\beta_{n}$ is worthwhile.
\begin{itemize}
	\item[1.] We have $\lambda_{n+1}^{qg}=\frac{\beta_{n}}{\beta_{n-1}}\lambda_{n}^g$. Multiplying the $n$ copies of these equations we get
	\begin{equation*}
		\prod_{k=1}^{n}\lambda_{k+1}^{qg}=\prod_{k=1}^{n}\frac{\beta_{k}}{\beta_{k-1}}\lambda_{k}^g,
	\end{equation*}
hence by \cite[Theorem 4.2]{Chihara book}, we have $\mathcal{L}^{QG}[(\mathcal{G}^Q_n(x;a))^2]=\lambda_{1}^{qg}\frac{\beta_{n}}{\beta_{0}}\mathcal{L}^G[(\mathcal{G}_{n-1}(x;a))^2]$. So we can write
\begin{equation*}
	\beta_{n}=\frac{\beta_{0}}{\lambda_{1}^{qg}}\frac{\mathcal{L}^{QG}[(\mathcal{G}^Q_n(x;a))^2]}{\mathcal{L}^G[(\mathcal{G}_{n-1}(x;a))^2]}.
\end{equation*}
Note that if $\beta_{0}=0,$ then $\beta_{n}=0$ for each $n\in \mathbb{N}$. Therefore $\beta_{0}\neq0$.
\item[2.] We have $c_{n+1}^{qg}=c_{n+1}^g+\beta_{n}-\beta_{n+1}$. Adding the $n$ copies of these equations we get
\begin{equation*}
	\sum_{k=0}^{n-1}c_{k+1}^{qg}=\sum_{k=0}^{n-1}c_{k+1}^g+\beta_{k}-\beta_{k+1}.
\end{equation*}
Hence by \cite[Theorem 4.2]{Chihara book}, we have
\begin{equation*}
	\beta_{n}=\beta_{0}-\text{coefficient of}~x^{n-1}~\text{in}~\mathcal{G}_n(x;a)+ \text{coefficient of}~x^{n-1}~\text{in}~\mathcal{G}^Q_n(x;a).
\end{equation*}
\item[3.] We can write (\ref{beta n restriction cond} ) as
\begin{align*}
	\beta_{n}-\beta_{n+1}+\frac{\lambda_{n}^g}{\beta_{n-1}}-\frac{\lambda_{n+1}^g}{\beta_{n}}=c_n^g-c_{n+1}^g.
\end{align*}
Adding the $n-1$ copies of the above equation, we get
\begin{align*}
	\beta_{2}-\beta_{n+1}+\frac{\lambda_{2}^g}{\beta_{1}}-\frac{\lambda_{n+1}^g}{\beta_{n}}&=c_2^g-c_{n+1}^g,\\
	\beta_{n+1}+\frac{\lambda_{n+1}^g}{\beta_{n}}&=C^{(1)}+c_{n+1}^g,\\
	\beta_{n}&=\frac{\lambda_{n+1}^g}{C^{(1)}+c_{n+1}^g-\beta_{n+1}},
\end{align*}
where $C^{(1)}=\beta_{2}+\frac{\lambda_{2}^g}{\beta_{1}}-c_2^g$. We can write $\beta_{n}$ in terms of a continued fraction
\begin{align}\label{Stieltjes CF of  beta n}
	\frac{\beta_{n}}{\lambda_{n+1}^g} = \frac{1}{C^{(1)}+c_{n+1}^g}\fminus\frac{\lambda_{n+2}^g}{C^{(1)}+c_{n+2}^g}\fminus\frac{\lambda_{n+3}^g}{C^{(1)}+c_{n+2}^g}\fminus\frac{\lambda_{n+3}^g}{C^{(1)}+c_{n+3}^g}\fminus\cdots
\end{align}
Note that for every fixed value of $n$, (\ref{Stieltjes CF of beta n}) is  a Stieltjes continued fraction. Following \cite[Equation 6.6]{Ismail_2019_quasi-orthogonal}, we obtain a sequence of orthogonal polynomials with respect to a measure $\tilde{\mu}^{(n)}$ for a fixed value of $n$ associated with the continued fraction (\ref{Stieltjes CF of beta n}). Therefore, expressing the continued fraction in terms of the Stieltjes transform of the measure $\tilde{\mu}^{(n)}$, we have
\begin{align}\label{Stieltjes integral for beta n}
\beta_{n}=\lambda_{n+1}^g\int_{-\infty}^{\infty}\frac{d\tilde{\mu}^{(n)}(z)}{C^{(1)}-z},\hspace{0.5cm}C^{(1)}\in \mathbb{C}\backslash supp(\tilde{\mu}).
\end{align}

\end{itemize}
 We have started with the sequence of monic orthogonal polynomials $\{P_n(x)\}_{n=0}^{\infty}$ with respect to the quasi-definite linear functional $\mathcal{L}$. Linear spectral transformations of $\mathcal{L}$ yield several sequences of orthogonal polynomials. From the Geronimus polynomials, we introduce the concept of quasi-Geronimus polynomial of order one. In this process, the orthogonality condition for the quasi-Geronimus polynomial of order one is relaxed. However, the previous result indicates that orthogonality can still be achieved with a suitable sequence of the constants $\beta_{n}$ in the definition the monic quasi-Geronimus polynomials.

  Since it may not be possible to get orthogonality for monic quasi-Geronimus polynomials for any choice on $\beta_{n}$, the next theorem proves that, even without specific conditions on $\beta_n$, we can recover the source orthogonal polynomial $P_n(x)$ from the monic quasi-Geronimus polynomials of order one. This recovery is achieved using different polynomials generated by spectral transformations, and the theorem specifically uses monic Geronimus polynomials for this purpose.
	\begin{theorem} Let $\mathcal{G}_{n}^{Q}(x;a_1)$ be a monic quasi-Geronimus polynomial of order one for some $a_1\in\mathbb{R}$. Let $\{\mathcal{G}_{n}(x;a_2)\}_{n=0}^{\infty}$ be a sequence of monic orthogonal polynomials with respect to $\mathcal{L}^G$ at $a_2\in\mathbb{R}$. Then there exist sequences of real numbers  $\{\gamma_n\}_{n=0}^{\infty}$ and $\{\eta_n\}_{n=0}^{\infty}$  such that
		\begin{align*}
			P_{n}(x)=\frac{1}{(x-\gamma_n)}\mathcal{G}_{n+1}^{Q}(x;a_1)+\frac{\eta_n}{(x-\gamma_n)}\mathcal{G}_n(x;a_2), ~n\geq0.
		\end{align*}
	\end{theorem}

\begin{proof} Consider
	\begin{align*}
		\mathcal{G}_n^G(a_1,a_2;x)&=\frac{1}{(x-\gamma_n)}\mathcal{G}_{n+1}^{Q}(x;a_1)+\frac{\eta_n}{(x-\gamma_n)}\mathcal{G}_n(x;a_2)\\
		&=\frac{1}{x-\gamma_n}\left[\mathcal{G}_{n+1}^{Q}(x;a_1)-(x-\gamma_n)P_n(x)+\eta_n\mathcal{G}_n(x;a_2)\right]+P_n(x).
	\end{align*}

Notice that
	\begin{align*}
		&\mathcal{G}_{n+1}^{Q}(x;a_1)-(x-\gamma_n)P_n(x)+\eta_n\mathcal{G}_n(x;a_2)\\&=\mathcal{G}_{n+1}(x;a_1)+\beta_{n+1}\mathcal{G}_n(x;a_1)-(x-\gamma_n)P_n(x)+\eta_nP_n(x)+\chi_n(a_2)\eta_nP_{n-1}(x)\\
		&=P_{n+1}(x)+\chi_{n+1}(a_1)P_n(x)+\beta_{n+1}P_n(x)+\beta_{n+1}\chi_n(a_1)P_{n-1}(x)-(x-\gamma_n)P_n(x)\\
		&\hspace{7cm}+\eta_nP_n(x)+\chi_n(a_2)P_{n-1}(x)\\
		&=P_{n+1}(x)-(x-\gamma_n-\chi_{n+1}(a_1)-\beta_{n+1}-\eta_n)P_n(x)+(\beta_{n+1}\chi_n(a_1)+\eta_n\chi_n(a_2))P_{n-1}(x).
	\end{align*}
If we choose
\begin{align*}
	\eta_n=\frac{\lambda_{n+1}-\beta_{n+1}\chi_n(a_1)}{\chi_n(a_2)},
\end{align*}
and
\begin{align*}
	\gamma_n=c_{n+1}-\chi_{n+1}(a_1)-\beta_{n+1}-\frac{\lambda_{n+1}-\beta_{n+1}\chi_n(a_1)}{\chi_n(a_2)},
\end{align*}

then from the TTRR satisfied by the polynomials $P_{n}(x)$ we get the desired result.
\end{proof}

The orthogonal polynomial $P_n(x)$ is obtained from the monic quasi-Geronimus polynomial of order one through a linear combination with the polynomials generated by Uvarov transformation. The process, as detailed in Theorem \ref{Recovery from quasi Geronimus and Uvarov}, highlights the necessity of three sequences of constants for the recovery of orthogonal polynomials.

\begin{theorem}\label{Recovery from quasi Geronimus and Uvarov}Let $\mathcal{G}_{n}^{Q}(x;a_1)$ be a quasi-Geronimus polynomial of order one for some $a_1\in\mathbb{R}$. Suppose $\{\mathcal{U}_{n}(x;a_2)\}_{n=0}^{\infty}$ be a sequence of monic Uvarov polynomials with respect to $\mathcal{L}^U$ at $a\in\mathbb{R}$. Then there exist sequences $\{\zeta_n\}_{n=0}^{\infty}$, $\{\gamma_n\}_{n=0}^{\infty}$ and $\{\eta_n\}_{n=0}^{\infty}$ such that
	\begin{align*}
	P_{n} (x)=	\frac{1}{\zeta_n(x-\eta_n)}\mathcal{G}_{n+1}^{Q}(x;a_1)+\frac{\gamma_n(x-a_2)}{\zeta_n(x-\eta_n)}\mathcal{U}_n(x;a_2), ~n\geq 0.
	\end{align*}
\end{theorem}
\begin{proof}
Let consider
	\begin{align*}
		\mathcal{G}_n^U(a_1,a_2;x)&=	\frac{1}{\zeta_n(x-\eta_n)}\mathcal{G}_{n+1}^{Q}(x;a_1)+\frac{\gamma_n(x-a_2)}{\zeta_n(x-\eta_n)}\mathcal{U}_n(x;a_2)\\
		&=\frac{1}{\zeta_n(x-\eta_n)}\left[\mathcal{G}_{n+1}^{Q}(x;a_1)-\zeta_n(x-\eta_n)P_n(x)+\gamma_n(x-a_2)\mathcal{U}_n(x;a_2)\right]+ P_n(x).
	\end{align*}
Thus
	\begin{align*}
		\mathcal{G}_{n+1}^{Q}(x;a_1)&-\zeta_n(x-\eta_n)P_n(x)+\gamma_n(x-a_2)\mathcal{U}_n(x;a_2)\\
		&=P_{n+1}(x)+(\chi_{n+1}(a_1)+\beta_{n+1})P_n(x)+\beta_{n+1}\chi_n(a_1)P_{n-1}(x)-\zeta_n(x-\eta_n)P_n(x)\\
		&+\gamma_n(x-a_2)P_n(x)-\gamma_ns_nP_n(x)+\gamma_ns_n\frac{P_n(a_2)}{P_{n-1}(a_2)}P_{n-1}(x).
	\end{align*}
Using the TTRR satisfied by $P_n(x)$ and combining the coefficients of $P_{n-1}$, $P_n$ and $P_{n+1}$, we can write the right hand side of the above equation as
\begin{align*}
	&\mathcal{G}_{n+1}^{Q}(x;a_1)-\zeta_n(x-\eta_n)P_n(x)+\gamma_n(x-a_2)\mathcal{U}_n(x;a_2)\\=
	&(1+\zeta_n+\gamma_n)P_{n+1}(x)+(\chi_{n+1}(a_1)+\beta_{n+1}-\zeta_n\lambda_{n+1}+\gamma_n\lambda_{n+1}-\gamma_ns_n+\zeta_n\eta_n-\gamma_na_2)P_{n}(x)\\
	&\left(\beta_{n+1}\chi_n(a_1)-\zeta_nc_{n+1}+\gamma_nc_{n+1}+\gamma_ns_n\frac{P_n(a_2)}{P_{n-1}(a_2)}\right)P_{n-1}(x).
\end{align*}
By setting the above equation equals zero and since $P_{n-1}$, $P_n$ and $P_{n+1}$ are linearly independent, the above equation vanishes if we choose the coefficients $\gamma_n$, $\zeta_n$ and $\eta_n$ as follows
\begin{align*}
	&\gamma_n=\left(c_{n+1}-\chi_n(a_1)\beta_{n+1}\right)\frac{P_{n-1}(a_2)}{s_nP_n(a_2)},\\
	&\zeta_n=1+\left(c_{n+1}-\chi_n(a_1)\beta_{n+1}\right)\frac{P_{n-1}(a_2)}{s_nP_n(a_2)}
\end{align*}
and
\begin{align*}
	\eta_n=\frac{1}{\zeta_n}\left[\lambda_{n+1}+(s_n+a_2)\left(c_{n+1}-\chi_n(a_1)\beta_{n+1}\right)\frac{P_{n-1}(a_2)}{s_nP_n(a_2)}-\beta_{n+1}-\chi_{n+1}(a_1)\right].
\end{align*}
This completes the proof.
\end{proof}
We recover the orthogonal polynomials $P_n(x)$ from the quasi-Geronimus polynomial of order one using polynomials generated by Geronimus and Uvarov transformations. In the subsequent theorem, it is also shown that obtaining the orthogonal polynomials from the Christoffel transformation is feasible.
\begin{theorem}Let $\mathcal{G}_{n}^{Q}(x;a_1)$ be a monic quasi-Geronimus polynomial of order one for some $a_1\in\mathbb{R}$. Let assume that $\{\mathcal{C}_{n}(x;a_2)\}_{n=0}^{\infty}$ is a sequence of kernel polynomials with respect to Christoffel transformation $\mathcal{L}^C$ at $a_2\in\mathbb{R}$. Then there exist sequences $\{\gamma_n\}_{n=0}^{\infty}$ and $\{\eta_n\}_{n=0}^{\infty}$ such that
	\begin{align*}
	P_{n}(x)=	\frac{1}{x-\eta_n}\mathcal{G}_{n+1}^Q(x;a_1)+\frac{\gamma_n(x-a_2)}{x-\eta_n}\mathcal{C}_{n-1}(x;a_2), ~n\geq 0.
	\end{align*}

\end{theorem}
\begin{proof}Let consider
	\begin{align*}
		\mathcal{G}_n^C(a_1,a_2;x)&=	\frac{1}{x-\eta_n}\mathcal{G}_{n+1}^Q(x;a_1)+\frac{\gamma_n(x-a_2)}{x-\eta_n}\mathcal{C}_{n-1}(x;a_2)\\
		&=\frac{1}{x-\eta_n}\left[\mathcal{G}_{n+1}^Q(x;a_1)-(x-\eta_n)P_n(x)+\gamma_n(x-a_2)\mathcal{C}_{n-1}(x;a_2)\right]+P_n(x).
	\end{align*}

Then
	\begin{align*}
		&\mathcal{G}_{n+1}^Q(x;a_1)-(x-\eta_n)P_n(x)+\gamma_n(x-a_2)\mathcal{C}_{n-1}(x;a_2)\\&=P_{n+1}(x)+\chi_{n+1}(a_1)P_n(x)+\beta_{n+1}P_n(x)+\beta_{n+1}\chi_n(a_1)P_{n-1}(x)\\
		&\hspace{6.5cm}-(x-\eta_n)P_n(x)+\gamma_nP_n(x)-\gamma_n\frac{P_n(a_2)}{P_{n-1}(a_2)}P_{n-1}(x)\\
		&=P_{n+1}(x)-(x-\eta_n-\gamma_n-\beta_{n+1}-\chi_{n+1}(a_1))P_n(x)\\
		&\hspace{8cm}+\left(\beta_{n+1}\chi_{n+1}(a_1)-\gamma_n\frac{P_n(a_2)}{P_{n-1}(a_2)}\right)P_{n-1}(x).
	\end{align*}
If we choose  $\gamma_n$ and $\eta_n$ as
\begin{align*}
	\gamma_n=\left(\beta_{n+1}\chi_{n+1}(a_1)-\lambda_{n+1}\right)\frac{P_{n-1}(a_2)}{P_n(a_2)},
\end{align*}
and
\begin{align*}
	\eta_n=c_{n+1}-\beta_{n+1}-\chi_{n+1}(a_1)+\left(\lambda_{n+1}-\beta_{n+1}\chi_{n+1}(a_1)\right)\frac{P_{n-1}(a_2)}{P_n(a_2)},
\end{align*}

 taking into account the TTRR that the polynomials $P_{n}(x)$ satisfy, then the result follows.
\end{proof}
\section{Recovery from quasi-Uvarov polynomial of order one}\label{Rec: quasi-Uvarov and recovery}
This section deals with the self-perturbation of Uvarov polynomials. We discuss the difference equation satisfied by the so-called quasi-Uvarov polynomial of order one, as well as its orthogonality. The section concludes by obtaining the source orthogonal polynomials from the quasi-Uvarov polynomial of order one.
	\begin{definition}
 A polynomial $p$ is said to be quasi-Uvarov polynomial of order one if it is of degree at most $n$ and satisfies
	\begin{align*}
		\mathcal{L}^U(x^kp(x))=0~~for~~ k=0,1,2,...,n-2.
	\end{align*}
\end{definition}
Since $\{\mathcal{U}_n(x;a)\}_{n=0}^{\infty}$ is a sequence of polynomials orthogonal with respect to $\mathcal{L}^U$, it is straightforward to observe that the monic Uvarov polynomials of degree $n+1$ and $n$ are monic quasi-Uvarov polynomials of order one. The subsequent result characterizes the quasi-Uvarov polynomial of order one as a self-perturbation of monic Uvarov polynomials.
\begin{lemma}
	A polynomial $\mathcal{U}_n^Q(x;a)$ is a quasi-Uvarov polynomial of degree at most $n$ and order one if and only if $\mathcal{U}_n^Q(x;a)$ can be written as
	\begin{align}\label{Non monic Quasi-Uvarov order 1}
			\mathcal{U}_n^Q(x;a)= \alpha_{n, n}\mathcal{U}_n(x;a)+\alpha_{n-1,n}\mathcal{U}_{n-1}(x;a),
	\end{align}
	where  $\alpha_{n-1,n}$ and $\alpha_{n,n}$ cannot be zero simultaneously.
\end{lemma}
\begin{proof} See \cite{Chihara book},
\end{proof}
Next, the representation of source orthogonal polynomial in terms of consecutive degree of monic Uvarov polynomials is discussed.
\begin{proposition}
	Let $\{\mathcal{U}_n(x;a)\}_{n\geq1}$ be a sequence of monic Uvarov orthogonal polynomials. Then $P_n(x)$ can be written as
	\begin{align*}
		D_n(x)P_n(x)=(x-a)\frac{t_nP_{n}(a)}{\lambda_{n+1}P_{n-1}(a)}\mathcal{U}_{n+1}(x;a)+(x-a)(x-a-t_{n+1})\mathcal{U}_{n}(x;a),
	\end{align*}
where $D_n(x)=\left((x-a-t_n)+\frac{t_nP_{n}(a)}{\lambda_{n+1}P_{n-1}(a)}(x-c_{n+1})\right)(x-a-t_{n+1})+\frac{t_nt_{n+1}}{\lambda_{n+1}}\frac{P_{n+1}(a)}{P_{n-1}(a)}.$
\end{proposition}
\begin{proof}
	The expansion of the kernel polynomial allows us to express the monic Uvarov polynomials as
	\begin{align}\label{Uvarov n degree in terms source poly only }
		\mathcal{U}_{n}(x;a)=\left(1-\frac{t_{n}}{x-a}\right)P_{n}(x)+\frac{t_{n}}{x-a}\frac{P_{n}(a)}{P_{n-1}(a)}P_{n-1}(x).
	\end{align}
Using  the TTRR satisfied by $P_n(x)$, we can write \ref{Uvarov n degree in terms source poly only } as
\begin{align}
	(x-a)\mathcal{U}_{n}(x;a)=\frac{-t_nP_{n}(a)}{\lambda_{n+1}P_{n-1}(a)} P_{n+1}(x)+\left((x-a-t_n)+\frac{t_nP_{n}(a)}{\lambda_{n+1}P_{n-1}(a)}(x-c_{n+1})\right)P_n(x).
\end{align}
The transfer matrix for $\mathcal{U}_{n+1}(x;a)$ and $\mathcal{U}_{n}(x;a)$ from $P_{n}(x)$ and $ P_{n-1}(x)$  is
	\[(x-a)\begin{pmatrix}
	\mathcal{U}_{n+1}(x;a)\\
	\mathcal{U}_{n}(x;a)
\end{pmatrix}=\mathcal{N}(x)\begin{pmatrix}
	P_{n+1}(x)\\
	P_{n}(x)
\end{pmatrix},\]
where
\[\mathcal{N}(x)=\begin{pmatrix}
	x-a-t_{n+1} & t_{n+1}\frac{P_{n+1}(a)}{P_{n}(a)}\\
\frac{-t_nP_{n}(a)}{\lambda_{n+1}P_{n-1}(a)} &(x-a-t_n)+\frac{t_nP_{n}(a)}{\lambda_{n+1}P_{n-1}(a)}(x-c_{n+1})
\end{pmatrix}.\]
Since $\mathcal{N}(x)$ is nonsingular, we have

\[D_n(x)\begin{pmatrix}
P_{n+1}(x)\\
P_{n}(x)	
\end{pmatrix}=(x-a)\mathcal{N}'(x)\begin{pmatrix}
\mathcal{U}_{n+1}(x;a)\\
\mathcal{U}_{n}(x;a)
\end{pmatrix},\]
where
\[\mathcal{N}'(x)=\begin{pmatrix}
(x-a-t_n)+\frac{t_nP_{n}(a)}{\lambda_{n+1}P_{n-1}(a)}(x-c_{n+1})	 & -t_{n+1}\frac{P_{n+1}(a)}{P_{n}(a)}\\
	\frac{t_nP_{n}(a)}{\lambda_{n+1}P_{n-1}(a)} &x-a-t_{n+1}
\end{pmatrix}.\]
This completes the proof.
\end{proof}
The subsequent next result deals with the expression of orthogonal polynomials $P_n(x)$ as a linear combination of two quasi-Uvarov polynomials of order one and consecutive degrees. This procedure involves polynomial coefficients of degrees at most three.
\begin{lemma}Let $\mathcal{U}^Q_n(x;a)$ be a monic quasi-Uvarov polynomial of order one, i.e., $\alpha _{n,n}=1, \alpha _{n-1,n}= \alpha_{n}.$ Then the monic polynomials $P_n(x)$ orthogonal  with respect to the linear functional $\mathcal{L}$ can be written as
	\begin{align*}
		\frac{w_n(x)}{x-a}P_n(x)=e_n(x)\mathcal{U}_{n+1}^Q(x;a)+\left((x-a-t_{n+1})\lambda_{n+1}-t_{n+1}\alpha_{n+1}\frac{P_n(a)}{P_{n-1}(a)}\right)\mathcal{U}_{n}^Q(x;a), ~n\geq0,
	\end{align*}
where $\displaystyle{e_n(x)=\alpha_n\left(1+\frac{t_n}{\lambda_{n}}\frac{P_{n-1}(a)}{P_{n-2}(a)}\right)}x+\frac{P_{n}(a)}{P_{n-1}(a)}-t_n\alpha_n, s_n(x)=(x-a-t_{n+1})(x-c_{n+1})+\alpha_{n+1}(x-a)-\frac{P_{n+1}(a)}{P_n(a)}t_{n+1}-t_n\alpha_{n+1}$ and \\
\begin{align*}
	w_n(x)=&-\left(x-a-t_n+\frac{t_n\alpha_n}{\lambda_{n}}\frac{P_{n-1}(a)}{P_{n-2}(a)}\right)\left(t_{n+1}\alpha_{n+1}\frac{P_n(a)}{P_{n-1}(a)}-(x-a-t_{n+1})\lambda_{n+1}\right)\\
	&+s_n(x)e_n(x).
\end{align*}
\end{lemma}
\begin{proof}
	We write
	\begin{align*}
		(x-a)\mathcal{U}_{n}^Q(x;a)&=(x-a-t_n)P_n(x)+\left(\alpha_n(x-a)+t_n\frac{P_n(a)}{P_{n-1}(a)}-t_n\alpha_n\right)P_{n-1}(x)\\
		&+t_n\alpha_n\frac{P_{n-1}(a)}{P_{n-2}(a)}P_{n-2}(x).
	\end{align*}
Since $\displaystyle{\lambda_{n}P_{n-2}(x)=(x-c_n)P_{n-1}(x)-P_n(x)}$ and combining the coefficients of $P_n(x)$ and $P_{n-1}(x)$ we get
\begin{align*}
	(x-a)\mathcal{U}_{n}^Q(x;a)=&\left(x-a-t_n+\frac{t_n\alpha_n}{\lambda_{n}}\frac{P_{n-1}(a)}{P_{n-2}(a)}\right)P_n(x)\\
	&+\left(\alpha_n\left(1+\frac{t_n}{\lambda_{n}}\frac{P_{n-1}(a)}{P_{n-2}(a)}\right)x+\frac{P_{n}(a)}{P_{n-1}(a)}-t_n\alpha_n\right)P_{n-1}(x).
\end{align*}
Denoting $\displaystyle{e_n(x):=\alpha_n\left(1+\frac{t_n}{\lambda_{n}}\frac{P_{n-1}(a)}{P_{n-2}(a)}\right)x+\frac{P_{n}(a)}{P_{n-1}(a)}-t_n\alpha_n}$, we can write
\begin{equation*}
	(x-a)\mathcal{U}_{n}^Q(x;a)=\left(x-a-t_n+\frac{t_n\alpha_n}{\lambda_{n}}\frac{P_{n-1}(a)}{P_{n-2}(a)}\right)P_n(x)+e_n(x)P_{n-1}(x).
\end{equation*}
Similarly we can use $\displaystyle{P_{n+1}(x)=xP_n(x)-c_{n+1}P_n(x)-\lambda_{n+1}P_{n-1}(x)}$ to obtain the expression of $\mathcal{U}_{n+1}^Q(x;a)$ as
\begin{align*}
	(x-a)\mathcal{U}_{n+1}^Q(x;a)=s_n(x)P_n(x)+\left(t_{n+1}\alpha_{n+1}\frac{P_n(a)}{P_{n-1}(a)}-(x-a-t_{n+1})\lambda_{n+1}\right)P_{n-1}(x),
\end{align*}
	where $\displaystyle{s_n(x)=(x-a-t_{n+1})(x-c_{n+1})+\alpha_{n+1}(x-a)+\frac{P_{n+1}(a)}{P_n(a)}t_{n+1}-t_n\alpha_{n+1}}$.
	The transfer matrix  from  $P_{n}(x)$ and $ P_{n-1}(x)$ to  $\mathcal{U}_{n+1}^Q(x;a)$ and $\mathcal{U}_{n}^Q(x;a)$ is
	\[(x-a)\begin{pmatrix}
		\mathcal{U}_{n+1}^{Q}(x;a)\\
		\mathcal{U}_{n}^{Q}(x;a)
	\end{pmatrix}=\mathcal{M}(x)\begin{pmatrix}
		P_{n}(x)\\
		P_{n-1}(x)
	\end{pmatrix},\]
where
\[\mathcal{M}(x)=\begin{pmatrix}
	s_n(x) & t_{n+1}\alpha_{n+1}\frac{P_n(a)}{P_{n-1}(a)}-(x-a-t_{n+1})\lambda_{n+1}\\
	x-a-t_n+\frac{t_n\alpha_n}{\lambda_{n}}\frac{P_{n-1}(a)}{P_{n-2}(a)} &e_n(x)
\end{pmatrix}.\]
Since $\mathcal{M}(x)$ is nonsingular then we have
	\begin{align}\label{Matrix form polynomial in terms of Quasi Uvarov}
		w_n(x)\begin{pmatrix}
P_{n}(x)\\
P_{n-1}(x)	
\end{pmatrix}=(x-a)\mathcal{M}'(x)\begin{pmatrix}
	\mathcal{U}_{n+1}^{Q}(x;a)	\\
	\mathcal{U}_{n}^{Q}(x;a)
\end{pmatrix},\end{align}
where
\[\mathcal{M}'(x)=\begin{pmatrix}
	 e_n(x)& (x-a-t_{n+1})\lambda_{n+1}-t_{n+1}\alpha_{n+1}\frac{P_n(a)}{P_{n-1}(a)}\\
	-x+a+t_n-\frac{t_n\alpha_n}{\lambda_{n}}\frac{P_{n-1}(a)}{P_{n-2}(a)} &s_n(x)
\end{pmatrix}.\]
This completes the proof.
\end{proof}
The difference equation satisfied by the monic quasi-Geronimus polynomial of order one requires coefficients up to quadratic degree. However, as demonstrated in the next theorem, having coefficients with quadratic degree is not enough to derive the difference equation for the monic quasi-Uvarov polynomial of order one. The degree of coefficients needed to obtain the difference equation for monic quasi-Uvarov polynomial of order one is, at most, twice the degree of coefficients in the difference equation of monic quasi-Geronimus polynomial of order one.

\begin{theorem}\label{Difference equation for Quasi-Uvarov}Let $\{P_n(x)\}_{n=0}^\infty$ and $\{\mathcal{U}_{n}(x;a)\}_{n=0}^{\infty}$ be sequences of orthogonal polynomials with respect to $\mathcal{L}$ and $\mathcal{L}^U$, respectively. Then the difference equation satisfied by monic quasi-Uvarov polynomials of order one is
	\begin{align*}
		w_n(x)\mathcal{U}_{n+2}^{Q}(x;a)&=\left(r_{n+1}(x)e_n(x)-s_{n+1}(x)\lambda_{n+1}y_n(x)\right)\mathcal{U}_{n+1}^{Q}(x;a)\\
		&\hspace{4.6cm}+\left(r_{n+1}(x)h_n(x)-s_{n+1}(x)s_n(x)\lambda_{n+1}\right)\mathcal{U}_{n}^{Q}(x;a), ~n\geq0,
	\end{align*}
where $\displaystyle{y_n(x)=-x+a+t_n-\frac{t_n\alpha_nP_{n-1}(a)}{\lambda_{n}P_{n-2}(a)},~ r_{n+1}(x)=s_{n+1}(x)(x-c_{n+1})-h_{n+1}(x)}$ \\
and
$\displaystyle{h_n(x)=(x-a-t_{n+1})\lambda_{n+1}-t_{n+1}\alpha_{n+1}\frac{P_n(a)}{P_{n-1}(a)}}$.
\end{theorem}
\begin{proof}
	We write	\begin{align*}(x-a)\mathcal{U}_{n+2}^{Q}(x;a)=(x-a-t_{n+2})&P_{n+2}(x)+t_{n+2}\alpha_{n+2}\frac{P_{n+1}(a)}{P_{n}(a)}P_{n}(x)\\
		&+\left(\alpha_{n+2}(x-a)+t_{n+2}\frac{P_{n+2}(a)}{P_{n+1}(a)}-t_{n+2}\alpha_{n+2}\right)P_{n+1}(x).
	\end{align*}
Next,  according to the TTRR  the expansion of $P_{n+2}(x)$ in terms of $P_{n+1}(x)$ and $P_n(x)$ yields
\begin{align*}
	(x-a)\mathcal{U}_{n+2}^{Q}(x;a)=s_{n+1}(x)P_{n+1}(x)-h_{n+1}(x)P_{n}(x).
\end{align*}
Using the TTRR that the polynomials $P_n(x)$ satisfy, we get
\begin{align*}
	(x-a)\mathcal{U}_{n+2}^{Q}(x;a)=(s_{n+1}(x)(x-c_{n+1})+h_{n+1}(x))P_n(x)-s_{n+1}(x)\lambda_{n+1}P_{n-1}(x).
\end{align*}
Using \eqref{Matrix form polynomial in terms of Quasi Uvarov}, we get the desired result.
\end{proof}
In Theorem \ref{orthogonality of quasi-Uvarov}, we can lower the degree of coefficients in the difference equation satisfied by the monic quasi-Uvarov polynomial of order one. This reduction enables us to establish the three-term recurrence relation by applying conditions to the choices of $\alpha_{n}$.
\begin{theorem}\label{orthogonality of quasi-Uvarov}
	Let $\mathcal{U}_{n}^{Q}(x;a)$ be a quasi-Uvarov polynomial of order one with parameter $\beta_{n}$ such that
	\begin{align}\label{alpha n restriction cond}
		\alpha_{n}(c_{n+1}^u-c_n^u+\alpha_{n}-\alpha_{n+1})+\frac{\alpha_{n}}{\alpha_{n-1}}\lambda_{n}^u-\lambda_{n+1}^u=0, ~ n\geq2.
	\end{align}
 Then the polynomials $\mathcal{U}_{n}^{Q}(x;a)$ satisfy the TTRR
	\begin{align}
		\mathcal{U}_{n+1}^{Q}(x;a)-(x-c_{n+1}^{qu})\mathcal{U}_{n}^{Q}(x;a)+\lambda_{n+1}^{qu}\mathcal{U}_{n-1}^{Q}(x;a)=0, ~ n\geq0,
	\end{align}
where the recurrence coefficients are given by
\begin{align*}
	\lambda_{n+1}^{qu}=\frac{\alpha_{n}}{\alpha_{n-1}}\lambda_{n}^u, ~~~~~~~c_{n+1}^{qu}=c_{n+1}^u+\alpha_{n}-\alpha_{n+1}.
\end{align*}
If $\lambda_{n+1}^{qu}\neq 0$, then according to Favard's theorem  $\{\mathcal{U}_{n}^{Q}(x;a)\}_{n=1}^{\infty}$ is a sequence of monic  orthogonal polynomials with respect to a quasi-definite linear functional. If $\lambda_{n+1}^{qu}> 0,$ the linear functional is positive definite.
\end{theorem}
\begin{proof}We simplify
	\begin{align*}
		&\mathcal{U}_{n+1}^{Q}(x;a)-(x-c_{n+1}^{qu})\mathcal{U}_{n}^{Q}(x;a)+\lambda_{n+1}^{qu}\mathcal{U}_{n-1}^{Q}(x;a)
		=\mathcal{U}_{n+1}(x;a)\\
		&-(x-c_{n+1}^{qu}-\alpha_{n+1})\mathcal{U}_{n}(x;a)-(\alpha_{n}(x-c_{n+1}^{qu})-\lambda_{n+1}^{qu})\mathcal{U}_{n-1}(x;a)+\lambda_{n+1}^{qu}\alpha_{n-1}\mathcal{U}_{n-2}(x;a).
	\end{align*}
Using the TTRR satisfied by monic Uvarov polynomials, we obtain
	\begin{align*}
		&\mathcal{U}_{n+1}^{Q}(x;a)-(x-c_{n+1}^{qu})\mathcal{U}_{n}^{Q}(x;a)+\lambda_{n+1}^{qu}\mathcal{U}_{n-1}^{Q}(x;a)\\
		&=(c_{n+1}^{qu}-c_{n+1}^u+\alpha_{n+1}-\alpha_{n})\mathcal{U}_n(x;a)+(\alpha_{n}c_{n+1}^{qu}+\lambda_{n+1}^{qu}-\lambda_{n+1}^{u}-\alpha_{n}c_n^u)\mathcal{U}_{n-1}(x;a).\\
		&+(\lambda_{n+1}^{qu}\alpha_{n-1}-\alpha_{n}\lambda_{n}^u)\mathcal{U}_{n-2}(x;a).
	\end{align*}
	Since $\mathcal{U}_n(x;a)$, $\mathcal{U}_{n-1}(x;a)$ and $\mathcal{U}_{n-2}(x;a)$ are linearly independent the right hand side of the above expression vanishes if and only if
	\begin{align*}
		\alpha_{n}(c_{n+1}^u-c_n^u+\alpha_{n}-\alpha_{n+1})+\frac{\alpha_{n}}{\alpha_{n-1}}\lambda_{n}^u-\lambda_{n+1}^u=0,
	\end{align*}
	as well as
	\begin{align*}
		\lambda_{n+1}^{qu}=\frac{\alpha_{n}}{\alpha_{n-1}}\lambda_{n}^u, ~~~~~~~c_{n+1}^{qu}=c_{n+1}^u+\alpha_{n}-\alpha_{n+1}.
	\end{align*}
Thus the statement follows.
If $\lambda_{n+1}^{qu}\neq 0$, then according to Favard's theorem  $\{\mathcal{U}_{n}^{Q}(x;a)\}_{n=1}^{\infty}$ is a sequence of monic  orthogonal polynomials with respect to a quasi-definite linear functional. If $\lambda_{n+1}^{qu}> 0,$ the linear functional is positive definite.
\end{proof}
\subsection{ An alternative  representation of $\alpha_{n}$} We discuss the different representation of $\alpha_{n}$ in a similar manner as we discussed in the subsection \ref{Different representation of beta n}.
\begin{itemize}
	\item[1.] We have $\lambda_{n+1}^{qu}=\frac{\alpha_{n}}{\alpha_{n-1}}\lambda_{n}^u$. Multiplying  $n$ copies of these equations we get
	\begin{equation*}
		\alpha_{n}=\frac{\alpha_{0}}{\lambda_{1}^{qu}}\frac{\mathcal{L}^{QU}[(\mathcal{U}^Q_n(x;a))^2]}{\mathcal{L}^U[(\mathcal{U}_{n-1}(x;a))^2]}.
	\end{equation*}
	Note that if $\alpha_{0}=0,$ then $\alpha_{n}=0$ for each $n\in \mathbb{N}$. Therefore $\alpha_{0}\neq0$.
	\item[2.] We have $c_{n+1}^{qu}=c_{n+1}^u+\alpha_{n}-\alpha_{n+1}$. Adding  $n$ copies of these equations we get
	\begin{equation*}
		\alpha_{n}=\alpha_{0}-\text{coefficient of}~x^{n-1}~\text{in}~\mathcal{U}_n(x;a)+ \text{coefficient of}~x^{n-1}~\text{in}~\mathcal{U}^Q_n(x;a).
	\end{equation*}
	\item[3.] We can write (\ref{beta n restriction cond}) as
	\begin{align*}
		\alpha_{n}-\alpha_{n+1}+\frac{\lambda_{n}^u}{\alpha_{n-1}}-\frac{\lambda_{n+1}^u}{\alpha_{n}}=c_n^u-c_{n+1}^u.
	\end{align*}
	Adding $n-1$ copies of the above equation, we get
	\begin{align*}
		\alpha_{n}&=\frac{\lambda_{n+1}^u}{C^{(2)}+c_{n+1}^u-\alpha_{n+1}},
	\end{align*}
	where $C^{(2)}=\alpha_{2}+\frac{\lambda_{2}^u}{\alpha_{1}}-c_2^u$. We can write $\alpha_{n}$ in terms of the continued fraction
	\begin{align}\label{Stieltjes CF of  alpha n}
		\frac{\alpha_{n}}{\lambda_{n+1}^u} = \frac{1}{C^{(2)}+c_{n+1}^u}\fminus\frac{\lambda_{n+2}^u}{C^{(2)}+c_{n+2}^u}\fminus\frac{\lambda_{n+3}^u}{C^{(2)}+c_{n+2}^u}\fminus\frac{\lambda_{n+3}^u}{C^{(2)}+c_{n+3}^u}\fminus\cdots
	\end{align}
	 Hence, we obtain a sequence of orthogonal polynomials with respect to the measure $\tilde{\nu}^{(n)}$ for a fixed value of $n$ associated with the continued fraction (\ref{Stieltjes CF of alpha n}). Therefore, we can write the above continued fraction in terms of the Stieltjes integral.
	 	\begin{align*}
		\alpha_{n}=\lambda_{n+1}^u\int_{-\infty}^{\infty}\frac{d\tilde{\nu}^{(n)}(z)}{C^{(2)}-z},\hspace{0.5cm}C^{(2)}\in \mathbb{C} \backslash supp(\tilde{\nu}^{(n)}).
	\end{align*}
	
\end{itemize}

In the next theorem, we recover the source orthogonal polynomials $P_n(x)$ from  a linear combination of the monic quasi-Uvarov polynomial of order one and the monic  polynomials generated by the Christoffel transformation.
\begin{theorem} Let $\mathcal{U}_{n}^{Q}(x;a_1)$ be a quasi-Uvarov polynomial of order one for some $a_1\in\mathbb{R}$. Let $\{\mathcal{C}_{n}(x;a_2)\}_{n=0}^{\infty}$ be a sequence of orthogonal polynomials with respect to $\mathcal{L}^C$ at $a_2\in\mathbb{R}$.  Then there exist sequences $\{\zeta_n\}_{n=0}^{\infty},$ $\{\gamma_n\}_{n=0}^{\infty}$ and $\{\eta_n\}_{n=0}^{\infty}$ such that
	\begin{align*}
	P_{n}(x)=	\frac{(x-\eta_n)(x-a_2)}{\zeta_nx-\gamma_n}\mathcal{C}_{n-1}(x;a_2)-\frac{x-a_1}{\zeta_nx-\gamma_n}\mathcal{U}_n^Q(x;a_1), ~n\geq 1.
	\end{align*}
\end{theorem}
\begin{proof}
Let consider
	
	\begin{align*}
		&\mathcal{U}_n^C(a_1,a_2;x)=	\frac{(x-\eta_n)(x-a_2)}{\zeta_nx-\gamma_n}\mathcal{C}_{n-1}(x;a_2)-\frac{x-a_1}{\zeta_nx-\gamma_n}\mathcal{U}_n^Q(x;a_1)\\
		&=\frac{-1}{\zeta_nx-\gamma_n}\left[(x-a_1)\mathcal{U}_n^Q(x;a_1)+(\zeta_nx-\gamma_n)P_{n-1}(x)-(x-\eta_n)(x-a_2)\mathcal{C}_{n-1}(x;a_2)\right]+P_{n-1}(x).
	\end{align*}
Thus
	\begin{align*}
		&(x-a_1)\mathcal{U}_n^Q(x;a_1)+(\zeta_nx-\gamma_n)P_{n-1}(x)-(x-\eta_n)(x-a_2)\mathcal{C}_{n-1}(x;a_2)\\
		&=(x-a_1)P_n(x)-t_n(x-a_1)\mathcal{C}_{n-1}(x;a_1)+\alpha_n(x-a_1)P_{n-1}(x)-t_{n-1}\alpha_n(x-a_1)\mathcal{C}_{n-2}(x;a_1)\\
		&+(\zeta_nx-\gamma_n)P_{n-1}(x)-(x-\eta_n)P_n(x)+(x-\eta_n)\frac{P_n(a_2)}{P_{n-1}(a_2)}P_{n-1}(x).
	\end{align*}
	Using the expression for $xP_{n-1}(x)$ from the TTRR and combining the coefficients of  $P_{n-2}$, $P_{n-1}$ and $P_{n}$ we obtain
	\begin{align*}
		&(x-a_1)\mathcal{U}_n^Q(x;a_1)+(\zeta_nx-\gamma_n)P_{n-1}(x)-(x-\eta_n)(x-a_2)\mathcal{C}_{n-1}(x;a_2)\\
		&= \left(\zeta_n+\alpha_n+\eta_n-a_1-t_n+\frac{P_n(a_2)}{P_{n-1}(a_2)}\right)P_n(x)+\left(t_n\frac{P_n(a_1)}{P_{n-1}(a_1)}+\alpha_nc_n-\alpha_na_1-t_{n-1}\alpha_n\right.\\
		&\left.+\zeta_nc_n-\gamma_n+c_n\frac{P_n(a_2)}{P_{n-1}(a_2)}-\eta_n\frac{P_n(a_2)}{P_{n-1}(a_2)}\right)P_{n-1}(x)+\left(t_{n-1}\alpha_n\frac{P_{n-1}(a_1)}{P_{n-2}(a_1)}+\zeta_n\lambda_{n}+\alpha_n\lambda_{n}\right.\\
		&\left.+\lambda_{n}\frac{P_n(a_2)}{P_{n-1}(a_2)}\right)P_{n-2}(x).
	\end{align*}
Since $P_{n-2}$, $P_{n-1}$ and $P_{n}$ are linearly independent the above expression vanishes by choosing $\zeta_n$, $\eta_n$ and $\gamma_n$ as follows
\begin{align*}
	&\zeta_n=\frac{1}{\lambda_{n}}\left[-\alpha_n\lambda_{n}-\lambda_{n}\frac{P_n(a_2)}{P_{n-1}(a_2)}-t_{n-1}\alpha_n\frac{P_{n-1}(a_1)}{P_{n-2}(a_1)}\right],\\
	&\eta_n=a_1+t_n-\alpha_n-\frac{P_n(a_2)}{P_{n-1}(a_2)}+\frac{1}{\lambda_{n}}\left[\alpha_n\lambda_{n}+\lambda_{n}\frac{P_n(a_2)}{P_{n-1}(a_2)}+t_{n-1}\alpha_n\frac{P_{n-1}(a_1)}{P_{n-2}(a_1)}\right]
\end{align*}
and
\begin{align*}
	\gamma_n=\zeta_nc_n-\eta_n\frac{P_n(a_2)}{P_{n-1}(a_2)}-\alpha_na_1-t_{n-1}\alpha_n+t_n\frac{P_n(a_1)}{P_{n-1}(a_1)}+\alpha_nc_n+c_n\frac{P_n(a_2)}{P_{n-1}(a_2)}.
\end{align*}
This completes the proof.
\end{proof}

 The next theorem addresses how to recover the orthogonal polynomials $P_n(x)$ through a linear combination of the monic quasi-Uvarov polynomials of order one and the monic polynomials generated by the Geronimus transformation.
\begin{theorem}Let $\mathcal{U}_{n}^{Q}(x;a_1)$ be a quasi-Uvarov polynomial of order one for some $a_1\in\mathbb{R}$. Further, suppose $\{\mathcal{G}_{n}(x;a_2)\}_{n=0}^{\infty}$ is a sequence of orthogonal polynomials with respect to $\mathcal{L}^G$ at $a_2\in\mathbb{R}$.Then there exist sequences $\{\zeta_n\}_{n=0}^{\infty},$ $\{\gamma_n\}_{n=0}^{\infty}$ and $\{\eta_n\}_{n=0}^{\infty}$ such that
	\begin{align*}
	P_{n}(x)=	\frac{x-\eta_n}{\zeta_nx-\gamma_n}\mathcal{G}_n(x;a_2)-\frac{x-a_1}{\zeta_nx-\gamma_n}\mathcal{U}_n^Q(x;a_1), ~n\geq0.
	\end{align*}

\end{theorem}
\begin{proof} 

Let consider
		\begin{align*}
		&\mathcal{U}_n^C(a_1,a_2;x)=	\frac{x-\eta_n}{\zeta_nx-\gamma_n}\mathcal{G}_n(x;a_2)-\frac{x-a_1}{\zeta_nx-\gamma_n}\mathcal{U}_n^Q(x;a_1)\\
		&=\frac{-1}{\zeta_nx-\gamma_n}\left[(x-a_1)\mathcal{U}_n^Q(x;a_1)+(\zeta_nx-\gamma_n)P_{n-1}(x)-(x-\eta_n)\mathcal{G}_n(x;a_2)\right]+P_{n-1}(x).
	\end{align*}
Then
	\begin{align*}
		&(x-a_1)\mathcal{U}_n^Q(x;a_1)+(\zeta_nx-\gamma_n)P_{n-1}(x)-(x-\eta_n)\mathcal{G}_n(x;a_2)\\
		&=(x-a_1)P_n(x)-t_n(x-a_1)\mathcal{C}_{n-1}(x;a_1)+\alpha_n(x-a_1)P_{n-1}(x)-t_{n-1}\alpha_n(x-a_1)\mathcal{C}_{n-2}(x;a_1)\\
		&+(\zeta_nx-\gamma_n)P_{n-1}(x)-(x-\eta_n)(P_n(x)+\chi_n(a_2)P_{n-1}(x))\\
		&=(x-a_1)P_n(x)-t_nP_n(x)+t_n\frac{P_n(a_1)}{P_{n-1}(a_1)}P_{n-1}(x)+\alpha_n(x-a_1)P_{n-1}(x)-t_{n-1}\alpha_nP_{n-1}(x)\\
		&+t_{n-1}\alpha_n\frac{P_{n-1}(a_1)}{P_{n-2}(a_1)}P_{n-2}(x)+(\zeta_nx-\gamma_n)P_{n-1}(x)-(x-\eta_n)(P_n(x)+\chi_n(a_2)P_{n-1}(x))\\
		&=(\eta_n-t_n-a_1)\left[P_n(x)-\left(\left(\frac{\chi_n(a_2)-\alpha_n-\zeta_n}{\eta_n-t_n-a_1}\right)x\right.\right.\\ &\left.\left.-\frac{\eta_n\chi_n(a_2)-\gamma_n+t_n\frac{P_n(a_1)}{P_{n-1}(a_1)}-t_{n-1}\alpha_n-\alpha_na_1}{\eta_n-t_n-a_1}\right)P_{n-1}(x)+\frac{t_n\alpha_n}{\eta_n-t_n-a_1}\frac{P_n(a_1)}{P_{n-1}(a_1)}P_{n-2}(x)\right].
	\end{align*}
By choosing $\eta_n$, $\gamma_n$ and  $\zeta_n$ as
\begin{align*}
	&\eta_n=\frac{1}{\lambda_{n}P_{n-1}(a_1)}\left[t_n\alpha_nP_n(a_1)+\lambda_{n}t_nP_{n-1}(a_1)+\lambda_{n}a_1P_{n-1}(a_1)\right],\\
	&\gamma_n=\eta_n\chi_n(a_2)-\alpha_na_1-t_{n-1}\alpha_n+t_n\frac{P_n(a_1)}{P_{n-1}(a_1)}-c_n(\eta_n-t_n-a_1),
\end{align*}
and
\begin{align*}
    \zeta_n=\chi_n(a_2)-\alpha_n+t_n+a_1-\frac{1}{\lambda_{n}P_{n-1}(a_1)}\left[t_n\alpha_nP_n(a_1)+\lambda_{n}t_nP_{n-1}(a_1)+\lambda_{n}a_1P_{n-1}(a_1)\right],
\end{align*}
we get
\begin{align*}
	(x-a_1)\mathcal{U}_n^Q(x;a_1)+(\zeta_nx-\gamma_n)P_{n-1}(x)&-(x-\gamma_n)\mathcal{G}_n(x;a_2)=(\eta_n-t_n-a_1)\left(P_n(x)\right.\\
	&\left.-(x-c_n)P_{n-1}(x)+\lambda_{n}P_{n-2}(x)\right).
\end{align*}
Since $P_n(x)$ satisfies the TTRR, hence we obtain the desired result.
\end{proof}
\section{Numerical experiments}\label{Numerical experiment}

Let $\{\mathcal{L}^{(\alpha)}_n(x)\}_{n=0}^{\infty}$ represent a sequence of monic Laguerre polynomials \cite{Chihara book}, defined by
\begin{align}
	\mathcal{L}^{(\alpha)}_n(x)=(-1)^n\Gamma(n+1)\sum_{j=0}^{n}\frac{(-1)^j}{\Gamma(j+1)} \left(\begin{array}{c}
		n+\alpha \\
		n-j
	\end{array}\right) x^j.
\end{align}

The monic Laguerre polynomials constitute an orthogonal set on the interval $(0,\infty)$ with respect to the weight function $w(x;\alpha)=x^{\alpha}e^{-x}$, $\alpha>-1$. These polynomials follow a three-term recurrence relation \cite[page 154]{Chihara book} given by

\begin{align}
	\mathcal{L}^{(\alpha)}_{n+1}(x)=(x-c_{n+1})\mathcal{L}^{(\alpha)}_{n}(x)-\lambda_{n+1}\mathcal{L}^{(\alpha)}_{n-1}(x),
\end{align}
with initial conditions $\mathcal{L}^{(\alpha)}_{-1}(x)=0, \mathcal{L}^{(\alpha)}_{0}(x)=1$. The recurrence relation is characterized by the parameters $c_{n+1}=2n+\alpha+1$ and $\lambda_{n+1}=n(n+\alpha)$.\\
\subsection{The Geronimus Case}
Upon applying the Geronimus transformation \eqref{Geronimus linear funtional} to the Laguerre linear functional with parameter $\alpha$ when $a=0$ and $M=\Gamma(\alpha)$, the resulting transformed weight is  $w(x;\alpha)=x^{\alpha-1}e^{-x}$ for $\alpha>0$, which corresponds to the Laguerre weight with parameter $\alpha-1.$ Consequently, the Geronimus polynomial is expressed as:

\begin{align}\label{Geronimus polynomial of Laguerre}
	\mathcal{G}_n(x;0):=\mathcal{L}^{(\alpha-1)}_n(x)=(-1)^n\Gamma(n+1)\sum_{j=0}^{n}\frac{(-1)^j}{\Gamma(j+1)} \left(\begin{array}{c}
		n+\alpha-1 \\
		n-j
	\end{array}\right) x^j.
\end{align}
This polynomial satisfies a TTRR given by:
\begin{align}
	\mathcal{G}_{n+1}(x;0)=(x-c^g_{n+1})\mathcal{G}_n(x;0)-\lambda^g_{n+1}\mathcal{G}_{n-1}(x;0),
\end{align}
where $c^g_{n+1}=2n+\alpha$ and $\lambda^g_{n+1}=n(n+\alpha-1)$.\\

The Geronimus polynomial \eqref{Geronimus polynomial of Laguerre} associated with the Laguerre weight can be decomposed and expressed as \eqref{Geronimus polynomial}. This decomposition is obtained by comparing coefficients and utilizing Pascal's rule. Through this process, we derive the result $\chi_n(0)=n$.\\

The monic quasi-Geronimus Laguerre polynomial of order one is given by
\begin{align}\label{quasi Geronimus Laguerre poly}
	\mathcal{G}^Q_{n+1}(x;0)=\mathcal{G}_{n+1}(x;0)+\beta_{n+1}\mathcal{G}_{n}(x;0)=\mathcal{L}^{(\alpha-1)}_{n+1}(x)+\beta_{n+1}\mathcal{L}^{(\alpha-1)}_n(x).
\end{align}
It is a well-known result from \cite[Theorem 5.2]{Chihara book} that at most one zero of a quasi-orthogonal polynomial of order one lies outside the support of the measure of orthogonality. Table \ref{zeros of quasi-Geronimus Laguerre} illustrates this behavior.

\begin{table}[ht]
	\begin{center}
	\resizebox{!}{1.6cm}{\begin{tabular}{|c|c|c|c|}
			\hline
			\multicolumn{4}{|c|}{Zeros of $\mathcal{G}^Q_{6}(x;0)$}\\
			\hline
			$\beta_n=0.5$, $\alpha=0.9$ &$\beta_n=1$, $\alpha=1.5$&$\beta_n=7$, $\alpha=1$ &$\beta_n=6$, $\alpha=0.1$\\
			\hline
			0.193294&0.355981&-0.248125&-0.0584409\\
			\hline
			1.11293&1.44484&0.475247&0.108916\\
			\hline
			2.86119&3.3362&1.96233&1.31668\\
			\hline
			5.57689&6.17578&4.45828&3.53458\\
			\hline
			9.5578&10.2685&8.24579&7.05567\\
			\hline
			15.5979&16.4187&14.1065&12.6426\\
			\hline
		\end{tabular}}
		\captionof{table}{Zeros of  $\mathcal{G}^Q_{6}(x;0)$}
		\label{zeros of quasi-Geronimus Laguerre}
	\end{center}
\end{table}
We see that for $\beta_n=0.5$ and $\alpha=0.9$, all zeros of $\mathcal{G}^Q_{6}(x;0)$ are within the support of the Laguerre weight. Similarly, for $\beta_n=1$ and $\alpha=1.5$, the zeros also lie within the support. However, when $\beta_{n}=7$ and $\alpha=1$, exactly one zero of $\mathcal{G}^Q_{6}(x;0)$ is outside the support, as it is shown in Table \ref{zeros of quasi-Geronimus Laguerre}. The same holds true for $\beta_{n}=6$ and $\alpha=0.1$. \\

Consider a specific choice for $\beta_{n}$, namely $\beta_n=n$. In this case, the quasi-Geronimus polynomial of order becomes the monic Laguerre polynomial of degree $n+1$ with parameter $\alpha-2$. Additionally, we can determine the recurrence coefficients required to express the difference equation in Theorem \ref{Difference equation for Quasi-Geronimus} for $\mathcal{G}^Q_{n+1}(x;0)$. These coefficients can be easily obtained by using the values of $c_n, \lambda_{n}$, and $\chi_{n}(0)$. Indeed,
\begin{align*}
	l_n(x)&=x-\alpha-n+\beta_{n+1},\\
	d_n(x)&=\beta_{n}\frac{x-n}{n+\alpha-1}+n,\\
	j_n(x)&=l_n(x)d_n(x)-(2n+\alpha+1+n(n+\alpha)-n\beta_{n+1})\left(\frac{\beta_{n}}{n+\alpha-1}-1\right),\\
	m_{n}(x)&=l_{n-1}(x)(x-2n-\alpha+1)+\beta_{n+1}n-n(n+\alpha).
\end{align*}
Therefore, the difference equation reads as
\begin{align*}
	j_n(x)&\mathcal{G}_{n+2}^{Q}(x;0)=\left(d_n(x)m_{n+1}(x)-n(n+\alpha)l_{n+1}(x)\left(\frac{\beta_n}{(n+\alpha-1)}-1\right)\right)\mathcal{G}_{n+1}^{Q}(x;0)\\
	&+\huge(m_{n+1}(x)(2n+\alpha+1+n(n+\alpha)-n\beta_{n+1})-n(n+\alpha)l_{n}(x)l_{n+1}(x))\huge\mathcal{G}_{n}^{Q}(x;0).
\end{align*}
In order to establish the orthogonality of the quasi-Geronimus polynomials of order one with parameter $\alpha-1$, it is necessary to reduce the degree of recurrence coefficients in the aforementioned difference equation. This reduction is achieved by calculating the sequence of constants $\beta_n$  such that \eqref{beta n restriction cond} holds. The specific condition is given by the equation
\begin{align}
\nonumber \beta_n(2+\beta_n-\beta_{n+1})+\frac{\beta_n}{\beta_{n-1}}(n-1)(n+\alpha-2)-n(n+\alpha-1)=0,\\
	(2+\beta_n-\beta_{n+1})+\frac{1}{\beta_{n-1}}(n-1)(n+\alpha-2)-\frac{1}{\beta_n}n(n+\alpha-1)=0.
\end{align}
Summing over $j=2$ to $n+1$, the equation becomes
\begin{align*}
	(2n+\beta_2-\beta_{n+2})+\frac{1}{\beta_{1}}\alpha-\frac{1}{\beta_{n+1}}(n+1)(n+\alpha)=0.
\end{align*}
By choosing $\beta_1=\alpha$ and $\beta_2=\alpha+1$, we recursively obtain $\beta_{n}=n+\alpha-1$. Thus, by Proposition \ref{orthogonality of quasi-Geronimus}, the monic quasi-Geronimus Laguerre polynomial of order one, denoted as $\mathcal{G}^Q_{n+1}(x;0)$, given by

 \begin{align}\label{quasi Geronimus orthogonal Laguerre}
 	\mathcal{G}^Q_{n+1}(x;0)=\mathcal{G}_{n+1}(x;0)+(n+\alpha)\mathcal{G}_{n}(x;0)=\mathcal{L}^{(\alpha-1)}_{n+1}(x)+(n+\alpha)\mathcal{L}^{(\alpha-1)}_n(x),
 \end{align}
satisfies the TTRR with  coefficients
\begin{align*}
	\lambda_{n+1}^{qg}=(n-1)(n+\alpha-1)>0,~~  c_{n+1}^{qg}=2n+\alpha-1.
\end{align*}


 As observed, the sequence of quasi-Geronimus Laguerre polynomial of order one is orthogonal when $\beta_{n+1}=n+\alpha$. Consequently, the zeros of $\mathcal{G}^Q_{n+1}(x;0)$ are within the interval $(0,\infty)$. In Table \ref{Zeros of quasi-Geronimus orthogonal Laguerre}, we illustrate that, for $\alpha=0.1$ and $\alpha=0.5$, the zeros of $\mathcal{G}^Q_{5}(x;0)$  lie within the interval $(0,\infty)$, with one zero positioned exactly on the boundary of the support. Additionally, Figure \ref{Interlace_Zeros_quasi-Geronimus orthogonal Laguerre} and Table \ref{Interlacing of quasi-Geronimus orthogonal Laguerre} show that for $\alpha=1$, the zeros of $\mathcal{G}^Q_{5}(x;0)$ and $\mathcal{G}^Q_{6}(x;0)$ interlace within the interval $(0, \infty)$.
\vspace{0.5cm}

\begin{small}
	\begin{minipage}[c]{0.4\textwidth}
		\hspace{-0.3cm}		\resizebox{!}{1.5cm}{\begin{tabular}{|c|c|}
					\hline
				\multicolumn{2}{|c|}{Zeros of $\mathcal{G}^Q_{5}(x;0)$}\\
				\hline
				$\alpha=0.1$, $n=4$, $\beta_{n+1}=n+\alpha$  &$\alpha=0.5$, $n=4$, $\beta_{n+1}=n+\alpha$ \\
				\hline
				0&0\\
				\hline
				0.36103&0.523526\\
				\hline
				1.8276&2.15665\\
				\hline
				4.65741&5.13739\\
				\hline
				9.55395&10.1824\\
				\hline
				-&-\\
				\hline
			\end{tabular}}
		\captionof{table}{Zeros of $\mathcal{G}^Q_{5}(x;0)$}
		\label{Zeros of quasi-Geronimus orthogonal Laguerre}
	\end{minipage}
\end{small}
\begin{minipage}[c]{0.4\textwidth}
	\hspace{1cm}\resizebox{!}{1.5cm}{\begin{tabular}{|c|c|}
				\hline
			\multicolumn{2}{|c|}{Zeros of $\mathcal{G}^Q_{n}(x;0)$}\\
			\hline
			$\alpha=1$, $n=5$, $\beta_{n}=n+\alpha-1$  &$\alpha=1$, $n=6$, $\beta_{n}=n+\alpha-1$ \\
				\hline
			0&0\\
			\hline
			0.74329&0.61703\\
			\hline
			2.57164&2.11297\\
			\hline
			5.73118&4.61083\\
			\hline
			10.9539&8.39907\\
			\hline
			-&14.2601\\
			\hline
	\end{tabular}}
	\captionof{table}{Zeros of $\mathcal{G}^Q_{n}(x;0)$}
	\label{Interlacing of quasi-Geronimus orthogonal Laguerre}
\end{minipage}
\begin{figure}[h!]
	\includegraphics[scale=1]{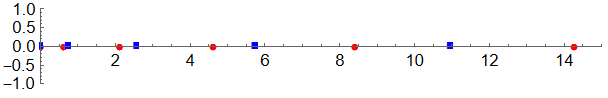}
	\caption{Zeros of $\mathcal{G}^Q_{5}(x;0)$ (blue squares) and $\mathcal{G}^Q_{6}(x;0)$ (red circles).}
	\label{Interlace_Zeros_quasi-Geronimus orthogonal Laguerre}
\end{figure}\\
The  particular case of Geronimus transformation applied to the Laguerre polynomials with parameter $\alpha$ described above yields  the Laguerre polynomials with parameter $\alpha-1$. In Figure \ref{Interlacing_Geronimus Laguerre and quasi-Geronimus} and Table \ref{Zeros_Laguerre_Quasi Geronimus Laguerre}, we show that, for $\alpha=1$, the zeros of $\mathcal{L}^{\alpha-1}_5(x)$ and $\mathcal{G}^Q_{5}(x;0)$, where $\beta_{n}=n+\alpha$, interlace.

 \begin{table}[ht]
 	\begin{center}
 	\resizebox{!}{1.7cm}{\begin{tabular}{|c|c|}
 			\hline
 			$\mathcal{G}^Q_{n+1}(x;0)$&$\mathcal{L}^{(\alpha-1)}_{n+1}(x)$\\
 			\hline
 			$n=4$, $\alpha=1$, $\beta_{n+1}=n+\alpha$ &$n=4$, $\alpha=1$\\
 			\hline
 			0&0.26356\\
 			\hline
 			0.743292&1.41340\\
 			\hline
 			2.57164&3.59643\\
 			\hline
 			5.73118&7.08581\\
 			\hline
 			10.9539&12.6408\\
 			\hline
 		\end{tabular}}
 		\captionof{table}{Interlacing of $\mathcal{G}^Q_{n+1}(x;0)$ and $\mathcal{L}^{(\alpha-1)}_{n+1}(x)$}
 		\label{Zeros_Laguerre_Quasi Geronimus Laguerre}
 	\end{center}
 \end{table}

\begin{figure}[h!]
	\includegraphics[scale=1]{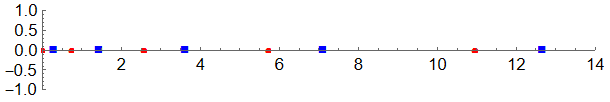}
	\caption{Zeros of $\mathcal{L}^{(0)}_{5}(x)$ (blue squares) and $\mathcal{G}^Q_{5}(x;0)$ (red circles).}
	\label{Interlacing_Geronimus Laguerre and quasi-Geronimus}
\end{figure}
For further analysis, we consider the monic case of the quasi-Geronimus polynomial of order one  $\mathcal{G}_n^Q(x;a)$ given in \eqref{Non monic Quasi-Geronimus order 1}, where $b_n=1$. Hence \eqref{Non monic Quasi-Geronimus order 1} can be rewritten as:
\begin{align}\label{monic quasi Geronimus poly}
	\mathcal{G}_n^Q(x;a)=\mathcal{G}_n(x;a)+k_n\mathcal{G}_{n-1}(x;a).
\end{align}
Note that $\beta_n$ is replaced by $k_n$, because specific condition on $\beta_n$ given by \eqref{beta n restriction cond}  turns the quasi-Geronimus polynomial of order one $\mathcal{G}_n^Q(x;a)$ to quasi-Geronimus orthogonal polynomial of order one, which we denote as $\widetilde{\mathcal{G}}_n^Q(x;a)$. Note that the polynomial $\widetilde{\mathcal{G}}_n^Q(x;a)$ satisfies  Proposition \ref{orthogonality of quasi-Geronimus}, with condition on $\beta_n$ given by \eqref{beta n restriction cond}. Consequently, the coefficient $\beta_{n+1}$ involved in \eqref{quasi Geronimus Laguerre poly} should be replaced by $k_{n+1}$, resulting in the modified expression:
\begin{align}\label{quasi Geronimus Laguerre poly para kn}
	\mathcal{G}^Q_{n+1}(x;0)=\mathcal{G}_{n+1}(x;0)+k_{n+1}\mathcal{G}_{n}(x;0)=\mathcal{L}^{(\alpha-1)}_{n+1}(x)+k_{n+1}\mathcal{L}^{(\alpha-1)}_n(x).
\end{align}
Similarly, we can rewrite quasi-Geronimus orthogonal Laguerre polynomial of order one \eqref{quasi Geronimus orthogonal Laguerre} as:
\begin{align}\label{quasi Geronimus orthogonal Laguerre poly}
	\widetilde{\mathcal{G}}^{Q}_{n+1}(x;0)=\mathcal{G}_{n+1}(x;0)+\beta_{n+1}\mathcal{G}_{n}(x;0)=\mathcal{L}^{(\alpha-1)}_{n+1}(x)+(n+\alpha)\mathcal{L}^{(\alpha-1)}_n(x).
\end{align}

We know that at most one zero of a quasi-orthogonal polynomial of order one lies outside the interval of orthogonality \cite[Theorem 5.2]{Chihara book}. Specifically, Table \ref{zeros of quasi-Geronimus Laguerre} illustrates that either exactly one zero of the $\mathcal{G}^{Q}_{n}(x;0)$ defined in \eqref{quasi Geronimus Laguerre poly para kn}  can be negative, or all zeros of the $\mathcal{G}^{Q}_{n}(x;0)$ defined in \eqref{quasi Geronimus Laguerre poly para kn} can be positive. However, it is still unclear what conditions are necessary for the coefficient $k_n$, as indicated in \eqref{monic quasi Geronimus poly}, to result in exactly one zero lying outside the interval of orthogonality. Similarly, it is unknown under what circumstances all zeros lie within this interval without transforming the quasi-Geronimus polynomial of order one $\mathcal{G}^{Q}_{n}(x;a)$ into an orthogonal system.

To address this question, we graphically analyze the polynomial $\mathcal{G}_n^Q(x;0)$ defined in \eqref{quasi Geronimus Laguerre poly para kn}. We provide two cases to elucidate this matter further:\\

\begin{small}
	\begin{minipage}[c]{0.4\textwidth}
			\resizebox{!}{1.4cm}{\begin{tabular}{|c|c|}
				\hline
				\multicolumn{2}{|c|}{$\alpha=2.5$, $n=5$}\\
				\hline
				Zeros of $\mathcal{G}^Q_{n}(x;0)$, $k_5=25$   &Zeros of $\widetilde{\mathcal{G}}^{Q}_{n}(x;0)$, $\beta_n=n+\alpha-1$ \\
				\hline
				-15.4435&0\\
				\hline
				1.02900&1.48624\\
				\hline
				3.13601&3.83768\\
				\hline
				6.60152&7.48206\\
				\hline
				12.1768&13.194\\
				\hline
				-&-\\
				\hline
		\end{tabular}}
	\captionof{table} {$\mathcal{G}^Q_{n}(x;0)$\eqref{quasi Geronimus Laguerre poly para kn} and $\widetilde{\mathcal{G}}^{Q}_{n}(x;0)$ \eqref{quasi Geronimus orthogonal Laguerre poly}.}
		\label{big negative zero_QGLP}
	\end{minipage}
\end{small}
\vspace{1cm}
\begin{minipage}[c]{0.4\textwidth}
	\hspace{1cm}\resizebox{!}{1.4cm}{\begin{tabular}{|c|c|}
			\hline
			\multicolumn{2}{|c|}{$\alpha=3$, $n=6$}\\
			\hline
			Zeros of $\mathcal{G}^Q_{n}(x;0)$, $k_6=10$  &Zeros of $\widetilde{\mathcal{G}}^Q_{n}(x;0)$, $\beta_{n}=n+\alpha-1$ \\
			\hline
			-0.966814&0\\
			\hline
			1.30868&1.49055\\
			\hline
			3.37753&3.58133\\
			\hline
			6.41475&6.627\\
			\hline
			10.728&10.9444\\
			\hline
			17.1378&17.3567\\
			\hline
	\end{tabular}}
	\captionof{table}{$\mathcal{G}^Q_{n}(x;0)$\eqref{quasi Geronimus Laguerre poly para kn} and $\widetilde{\mathcal{G}}^Q_{n}(x;0)$ \eqref{quasi Geronimus orthogonal Laguerre poly}.}
	\label{small negative zero_QGLP}
\end{minipage}
 \begin{figure}[h!]
\centering
\begin{subfigure}{.4\textwidth}
	\hspace{-1.2cm}
	\includegraphics[width=1.2\linewidth]{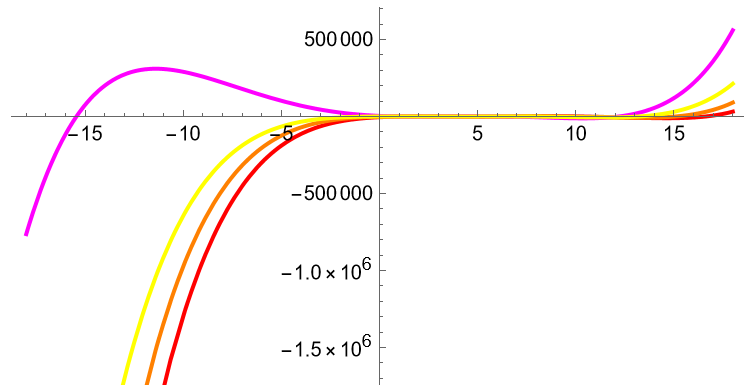}
	\caption{Graph of $\mathcal{L}_n^{(\alpha)}(x)$ (red), $\mathcal{L}_n^{(\alpha-1)}(x)$ (orange), $\mathcal{G}_n^{Q}(x;0)$   (magenta) and $\widetilde{\mathcal{G}}_n^Q(x;0)$  (yellow). $n=5, \alpha=2.5, k_5=25$.}
	\label{large gap b/w kn and n plus alpha-1}
\end{subfigure}
\hfil
\begin{subfigure}{.4\textwidth}
	\hspace{-0.7cm}
	\includegraphics[width=1.2\linewidth]{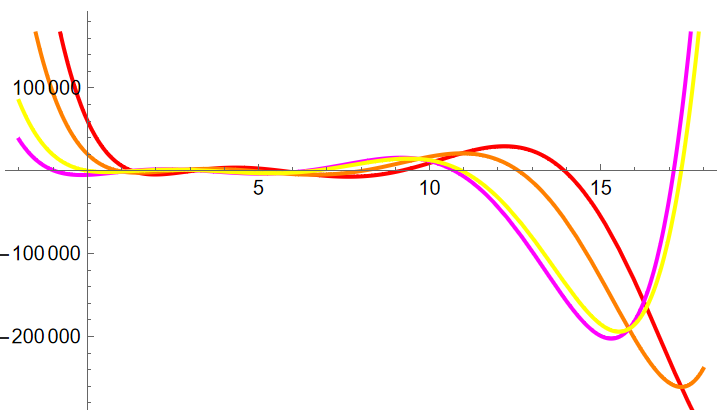}
	\caption{Graph of $\mathcal{L}_n^{(\alpha)}(x)$ (red), $\mathcal{L}_n^{(\alpha-1)}(x)$ (orange), $\mathcal{G}_n^{Q}(x;0)$  (magenta) and $\widetilde{\mathcal{G}}_n^Q(x;0)$  (yellow). $n=6, \alpha=3, k_6=10$.}
	\label{small gap b/w kn and n plus alpha-1}
\end{subfigure}
\caption{$\mathcal{G}_n^{Q}(x;0)$ \eqref{quasi Geronimus Laguerre poly para kn} and $\widetilde{\mathcal{G}}_n^Q(x;0)$ \eqref{quasi Geronimus orthogonal Laguerre poly}. For $k_n>n+\alpha-1$.}
\label{growth behavior when kn greater n plus alpha -1}
\end{figure}

\textbf{Case 1: when} $\bf{k_n>n+\alpha-1}$. Figure \ref{growth behavior when kn greater n plus alpha -1} illustrates that when the coefficient  $k_n$ given in \eqref{quasi Geronimus Laguerre poly para kn} exceeds $n+\alpha-1$, for any fixed degree, the growth of the polynomial $\mathcal{G}_n^{Q}(x;0)$ (magenta) defined in \eqref{quasi Geronimus Laguerre poly para kn} surpasses the growth of the polynomial $\widetilde{\mathcal{G}}_n^Q(x;0)$ (yellow) defined in \eqref{quasi Geronimus orthogonal Laguerre poly} beyond the last intersection point of the corresponding polynomials. It is worth mentioning that a negative zero of the polynomial $\mathcal{G}_n^{Q}(x;0)$ arises only when $k_n$ exceeds $n+\alpha-1$. If $k_n$ is significantly larger than $n+\alpha-1$, this negative zero moves away from the first zero (i.e., $x=0$) of the polynomial $\widetilde{\mathcal{G}}_n^Q(x;0)$ at a faster rate. Whereas, if $k_n$ is greater but closer to $n+\alpha-1$, the negative zero remain in the vicinity of the origin. This phenomenon is illustrated in Table \ref{big negative zero_QGLP} and Table \ref{small negative zero_QGLP} for $n=5,6$.  However, the scenario changes entirely when we examine the growth of these two polynomials before reaching the first intersection point. For odd-degree polynomials, the graph of $\widetilde{\mathcal{G}}_n^Q(x;0)$ (yellow) always remains below the graph of $\mathcal{G}_n^{Q}(x;0)$ (magenta) prior to the first intersection point of the polynomials, as depicted in Figure \ref{large gap b/w kn and n plus alpha-1}. On the other hand, for even degrees, the graph of $\widetilde{\mathcal{G}}_n^Q(x;0)$ (yellow) grows faster than the polynomial $\mathcal{G}_n^{Q}(x;0)$ (magenta) before reaching the first intersection point of the corresponding polynomials, as demonstrated in Figure \ref{small gap b/w kn and n plus alpha-1}.\\

\begin{small}
	\begin{minipage}[c]{0.4\textwidth}
		\resizebox{!}{1.25cm}{\begin{tabular}{|c|c|}
				\hline
				\multicolumn{2}{|c|}{$\alpha=4.5$, $n=5$}\\
				\hline
				Zeros of $\mathcal{G}^Q_{n}(x;0)$, $k_5=3$   &Zeros of $\widetilde{\mathcal{G}}^{Q}_{n}(x;0)$, $\beta_n=n+\alpha-1$ \\
				\hline
				1.54164&0\\
				\hline
				3.5791&2.61181\\
				\hline
				6.54479&5.56339\\
				\hline
				10.7668&9.76859\\
				\hline
				17.0676&16.0562\\
				\hline
		\end{tabular}}
		\captionof{table} {$\mathcal{G}^Q_{n}(x;0)$\eqref{quasi Geronimus Laguerre poly para kn} and $\widetilde{\mathcal{G}}^{Q}_{n}(x;0)$ \eqref{quasi Geronimus orthogonal Laguerre poly}.}
		\label{1 Zeros QGLP positive kn less n plus alpha-1}
	\end{minipage}
\end{small}
\vspace{1cm}
\begin{minipage}[c]{0.4\textwidth}
	\hspace{1cm}\resizebox{!}{1.25cm}{\begin{tabular}{|c|c|}
			\hline
			\multicolumn{2}{|c|}{$\alpha=3$, $n=4$}\\
			\hline
			Zeros of $\mathcal{G}^Q_{n}(x;0)$, $k_4=2$  &Zeros of $\widetilde{\mathcal{G}}^{Q}_{n}(x;0)$, $\beta_{n}=n+\alpha-1$ \\
			\hline
			1.07821&0\\
			\hline
			3.05525&2.14122\\
			\hline
			6.29734&5.31552\\
			\hline
			11.5692&10.5433\\
			\hline
			-&-\\
			\hline
	\end{tabular}}
	\captionof{table}{$\mathcal{G}^Q_{n}(x;0)$\eqref{quasi Geronimus Laguerre poly para kn} and $\widetilde{\mathcal{G}}^{Q}_{n}(x;0)$ \eqref{quasi Geronimus orthogonal Laguerre poly}.}
	\label{2 Zeros QGLP positive kn less n plus alpha-1}
\end{minipage}

\begin{figure}[h!]
	\begin{subfigure}{.4\textwidth}
		\hspace{-0.5cm}
		\includegraphics[width=1.2\linewidth]{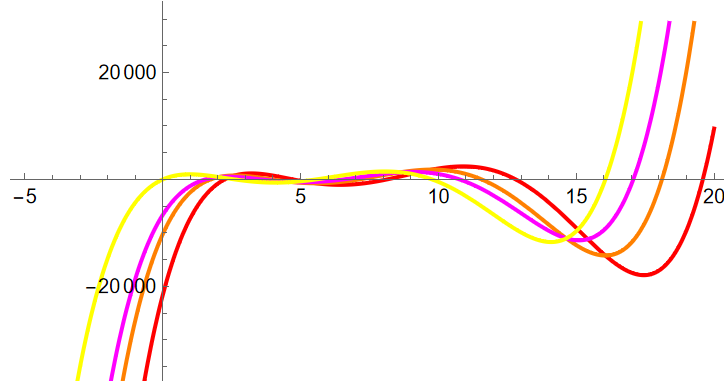}
		\caption{Graph of $\mathcal{L}_n^{(\alpha)}(x)$ (red), $\mathcal{L}_n^{(\alpha-1)}(x)$ (orange), $\mathcal{G}_n^{Q}(x;0)$ (magenta) and $\widetilde{\mathcal{G}}_n^Q(x;0)$ (yellow). $n=5, \alpha=4.5, k_5=3$.}
		\label{orthogonality domination 2}
	\end{subfigure}
	\hfil
	\begin{subfigure}{.4\textwidth}
		\hspace{-1cm}	\includegraphics[width=1.2\linewidth]{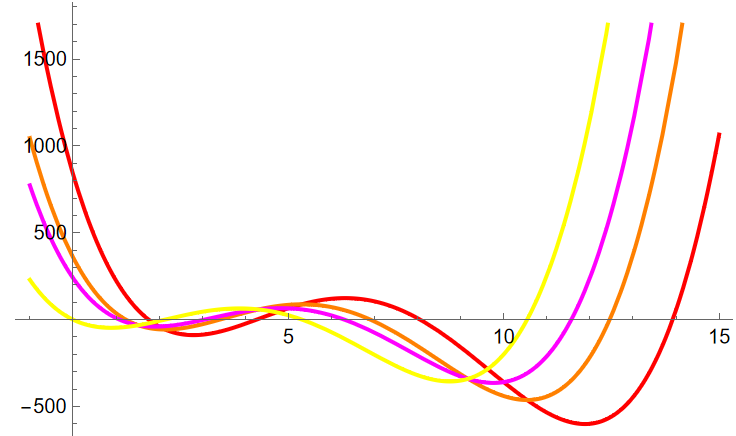}
		\caption{Graph of $\mathcal{L}_n^{(\alpha)}(x)$ (red), $\mathcal{L}_n^{(\alpha-1)}(x)$ (orange), $\mathcal{G}_n^{Q}(x;0)$ (magenta) and $\widetilde{\mathcal{G}}_n^Q(x;0)$ (yellow). $n=4, \alpha=3, k_4=2$.}
		\label{orthogonality domination 1}
	\end{subfigure}
		\caption{$\mathcal{G}_n^{Q}(x;0)$ \eqref{quasi Geronimus Laguerre poly para kn} and $\widetilde{\mathcal{G}}_n^Q(x;0)$ \eqref{quasi Geronimus orthogonal Laguerre poly}. For $k_n<n+\alpha-1$.}
	\label{growth behavior when kn less n plus alpha -1}
\end{figure}

\textbf{Case 2: when} $\bf{k_n<n+\alpha-1}$. The situation differs significantly compared to Case 1. It is evident from Table \ref{1 Zeros QGLP positive kn less n plus alpha-1} and Table \ref{2 Zeros QGLP positive kn less n plus alpha-1} that no negative zero of the polynomial $\mathcal{G}_n^Q(x;0)$ defined in \eqref{quasi Geronimus Laguerre poly para kn} occurs when $\beta_n<n+\alpha-1$. In other words, all the zeros of the polynomial $\mathcal{G}_n^Q(x;0)$  lie within the interval $(0,\infty)$. For any fixed degree, Table \ref{1 Zeros QGLP positive kn less n plus alpha-1} and Table \ref{2 Zeros QGLP positive kn less n plus alpha-1} illustrate that the first zero of the polynomial $\mathcal{G}_n^Q(x;0)$  does not fluctuate significantly. As expected from Case 1, the polynomial $\widetilde{\mathcal{G}}_n^Q(x;0)$ (yellow) tends to dominate the polynomial $\mathcal{G}_n^Q(x;0)$  (magenta) after the last intersection point when $k_n<n+\alpha-1$. For odd degrees, the polynomial $\mathcal{G}_n^Q(x;0)$ (magenta) defined in \eqref{quasi Geronimus Laguerre poly para kn} decays faster than the polynomial $\widetilde{\mathcal{G}}_n^Q(x;0)$ (yellow) defined in \eqref{quasi Geronimus orthogonal Laguerre poly} before reaching the first intersection point of the polynomials, as illustrated in Figure \ref{orthogonality domination 2}. On the other hand, for even degrees, the growth of the polynomial $\mathcal{G}_n^Q(x;0)$ (magenta) exceeds that of $\widetilde{\mathcal{G}}_n^Q(x;0)$(yellow)  before reaching the first intersection point of the polynomials, as demonstrated in Figure \ref{orthogonality domination 1}.\\

In this subsection, $\text{Mathematica}^{\text{\textregistered}}$  software is used to compute the zeros and graphically illustrate the zeros with interlacing properties.
\subsection{The Uvarov case}
When we apply the Christoffel transformation to the Laguerre weight $w(x;\alpha)=x^\alpha e^{-x}, \alpha>-1,$ with $a=0$, the transformed weight becomes $w(x;\alpha)=x^{\alpha+1} e^{-x}, \alpha>-1$. The corresponding Christoffel polynomial is given by
\begin{align*}
	\mathcal{C}_n(x;0):=\mathcal{L}^{(\alpha+1)}_n(x)=(-1)^n\Gamma(n+1)\sum_{j=0}^{n}\frac{(-1)^j}{\Gamma(j+1)} \left(\begin{array}{c}
		n+\alpha+1 \\
		n-j
	\end{array}\right) x^j.
\end{align*}
Moreover, from \eqref{Other form of kernel polynomial}, we can calculate the monic kernel polynomials at $x=0$ as
\begin{align}\
	\mathcal{K}_{n-1}(0,0)=\frac{\mathcal{L}^{(\alpha)}_{n-1}(0)}{\lambda_1\lambda_2...\lambda_{n}}\mathcal{C}_{n-1}(0;0)=\left(\begin{array}{c}
		n+\alpha \\
		n-1
	\end{array}\right).
\end{align}
The Uvarov transformation of the Laguerre weight at $a=0$ and $M=1$ is given by
\begin{align*}
	w^u(x;0)=x^{\alpha}e^{-x}+\delta(x-0).
\end{align*}
The corresponding Uvarov polynomial  can be expressed as
\begin{align*}
	\mathcal{U}_n(x;0)=\mathcal{L}^{(\alpha)}_n(x)-t_n\mathcal{L}^{(\alpha+1)}_n(x), ~n\geq 1,
\end{align*}
where $t_n$ can be calculated as follows
\begin{align*}
	t_n=\frac{\mathcal{L}^{(\alpha)}_n(0)\mathcal{L}^{(\alpha)}_{n-1}(0)}{\lambda_1\lambda_2...\lambda_{n}\left(1 + \mathcal{K}_{n-1}(0,0)\right)}=-(\alpha+1)\frac{\left(\begin{array}{c}
			n+\alpha \\
			n-1
		\end{array}\right)}{1+\left(\begin{array}{c}
		n+\alpha \\
		n-1
	\end{array}\right)}.
\end{align*}
The recurrence coefficients for the sequence of quasi-Uvarov polynomials of order one can be deduced  using $\lambda_n, c_n$, and $t_n.$  In such a way the sequence satisfies the difference equation obtained in Theorem \ref{Difference equation for Quasi-Uvarov}. The specific recurrence coefficients can be deduced as follows:
The quasi-Uvarov polynomial of order one is given by
\begin{align}
	\mathcal{U}_n^Q(x;0)=\mathcal{U}_n(x;0)+\alpha_n\mathcal{U}_{n-1}(x;0).
\end{align}
The recurrence coefficients for the sequence of quasi-Uvarov polynomials of order one can be computed by using $\lambda_n, c_n$, and $t_n.$, In such a way  the sequence satisfies the difference equation obtained in Theorem \ref{Difference equation for Quasi-Uvarov}. The specific recurrence coefficients are

\begin{align*}
	e_n(x)&=\alpha_n\left(1-\frac{nt_n}{n-1}\right)x-n-\alpha,\\
	s_n(x)&=(x-t_{n+1})(x-2n-\alpha-1)+x\alpha_{n+1}+(n+\alpha+1)t_{n+1}-t_n\alpha_{n+1},\\
	w_n(x)&=\left(x-t_n+\frac{\alpha t_n}{n-1}\right)((x-t_{n+1})n+t_{n+1}\alpha_{n+1})(n+\alpha)+s_n(x)e_n(x),\\
	y_n(x)&=-x+\frac{t_n(n-1+\alpha_n)}{n-1},\\
	h_n(x)&=(nx-nt_{n+1}+t_{n+1}\alpha_{n+1})(n+\alpha),\\
	r_n(x)&=s_n(x)(x-2n-\alpha+1)-(nx-nt_{n+1}+t_{n+1}\alpha_{n+1})(n+\alpha).
\end{align*}
Therefore, the difference equation satisfied by the quasi-Uvarov polynomial of order one is
\begin{align*}
	w_n(x)\mathcal{U}_{n+2}^{Q}(x;0)&=\left(r_{n+1}(x)e_n(x)-n(n+\alpha)s_{n+1}(x)y_n(x)\right)\mathcal{U}_{n+1}^{Q}(x;0)\\
	&\hspace{4cm}+\left(r_{n+1}(x)h_n(x)-n(n+\alpha)s_{n+1}(x)s_n(x)\right)\mathcal{U}_{n}^{Q}(x;0).
\end{align*}


\section{Concluding remarks}\label{Conclusion}
In this contribution we have focused the attention on quasi-orthogonal polynomials of order one associated with sequences of orthogonal polynomials defined by linear spectral transformations (Geronimus, Uvarov) of a given sequence of orthogonal polynomials. Recurrence relations for such quasi-orthogonal polynomials have been obtained in the direction analyzed in \cite{Ismail_2019_quasi-orthogonal} by using transfer matrices. On the other hand, we can recover our initial sequence of orthogonal polynomials from two quasi-Geronimus and two quasi-Uvarov sequences of polynomials, respectively.  We obtain a representation of our initial sequence of orthogonal polynomials by using quasi-Geronimus and Geronimus, quasi-Geronimus and Uvarov, quasi-Geronimus and Christoffel sequences of polynomials, respectively, Finally, the same procedure also holds in order to  recover our initial sequence of orthogonal polynomials  by using quasi-Uvarov and Christoffel, quasi-Uvarov and Geronimus sequences of polynomials.\\

We derived the closed form of $\beta_{n}$ satisfying \eqref{beta n restriction cond}, necessary for the orthogonality of the quasi-Geronimus Laguerre polynomials of order one $\mathcal{G}_n^Q(x;a)$. Specifically, the recurrence parameters and the three-term recurrence relation satisfied by the quasi-Geronimus Laguerre polynomial of order one are obtained. Moreover, it would be of interest to determine, if possible, the explicit form of $\alpha_n$ satisfying \eqref{alpha n restriction cond} and to recover the recurrence coefficients that ensure the existence of an orthogonality measure for the quasi-Uvarov Laguerre polynomial of order one. Finally, drawing from the numerical experiments conducted in Section \ref{Numerical experiment}, we conclude this manuscript by summarizing the following observations::

\begin{observation} Let $\mathcal{G}_n^Q(x;a)$ be the polynomial of degree $n$ as defined in \eqref{monic quasi Geronimus poly}, with a non-zero unknown parameter $k_n$. Let $\widetilde{\mathcal{G}}_n^Q(x;a)$ represent the degree $n$ monic quasi-Geronimus orthogonal polynomial of order one, with the known parameter $\beta_n$ provided in Proposition \ref{orthogonality of quasi-Geronimus}. If $\mathcal{G}_n^Q(x;a)$ and $\widetilde{\mathcal{G}}_n^Q(x;a)$ intersect exactly at $m$ points, which are ordered as
	\begin{align*}
		x_1<x_2<...<x_m,
	\end{align*}
	then for any fixed degree $n$ and for any $k_n>\beta_n$, we have
	\begin{align*}
		\widetilde{\mathcal{G}}_n^Q(x;a)<\mathcal{G}_n^Q(x;a),
	\end{align*}
	whenever $x>x_m$. Moreover, for $n=2k$,
	\begin{align*}
		\widetilde{\mathcal{G}}_{2k}^{Q}(x;a)>\mathcal{G}_{2k}^Q(x;a),
	\end{align*}
	and for   $n=2k-1$,
	\begin{align*}
		\widetilde{\mathcal{G}}_{2k}^{Q}(x;a)<\mathcal{G}_{2k}^Q(x;a),
	\end{align*}
	whenever $x<x_1$.
\end{observation}
\begin{observation}
Let $\mathcal{G}_n^Q(x;a)$ denote the monic quasi-Geronimus polynomial of order one, as defined in \eqref{monic quasi Geronimus poly}, with a free non-zero parameter $k_n$. The coefficient $\beta_n$ is given by Proposition \ref{orthogonality of quasi-Geronimus}. Then the following statements hold:
	\begin{enumerate}
		\item If $k_n>\beta_n$, then exactly one zero of $\mathcal{G}_n^Q(x;a)$ lies outside the interval of orthogonality.
		\item If $k_n<\beta_n$, then all the zeros of $\mathcal{G}_n^Q(x;a)$ lies inside the interval of orthogonality.
	\end{enumerate}
\end{observation}



	{\bf Acknowledgments}.  The work of the second author (FM) has been supported by FEDER/Ministerio de Ciencia e Innovación-Agencia Estatal de Investigación of Spain, grant PID2021-122154NB-I00, and the Madrid Government (Comunidad de Madrid-Spain) under the Multiannual Agreement with UC3M in the line of Excellence of University Professors, grant EPUC3M23 in the context of the V PRICIT (Regional Program of Research and Technological Innovation). The work of the third author (AS) has been supported by the Project No. NBHM/RP-1/2019 of National Board for Higher Mathematics (NBHM), DAE, Government of India.

\end{document}